\documentclass[11pt]{amsart}

\usepackage{times}
\usepackage{amsmath,amsfonts,amssymb,amsopn,amscd,amsthm}
\usepackage{graphicx,xspace}
\usepackage{epsfig}
\usepackage{enumitem}
\usepackage{xcolor}
\usepackage{bm}

\usepackage[T1]{fontenc}
\usepackage[utf8]{inputenc}

\usepackage{hyperref}

\theoremstyle{plain}
\newtheorem{lemma}{Lemma}[section]
\newtheorem{theorem}[lemma]{Theorem}
\newtheorem{corollary}[lemma]{Corollary}
\newtheorem{proposition}[lemma]{Proposition}

\newtheorem{conjecture}[lemma]{Conjecture}

\theoremstyle{remark}
\newtheorem{remark}[lemma]{Remark}
\newtheorem{example}[lemma]{Example}
\newtheorem{definition}[lemma]{Definition}

\newcommand{\QQ}{\mathbb{Q}}
\newcommand{\ring}{\mathbb{Z}}
\newcommand{\N}{\mathbb{N}}

\newcommand{\E}{\mathbb{E}}

\newcommand{\nbc}{\kappa}
\newcommand{\cl}{\mathcal{C}}
\newcommand{\pcl}{\mathcal{PC}}
\newcommand{\Young}{\mathcal{Y}}
\newcommand{\A}{\mathcal{A}}
\newcommand{\B}{\mathcal{B}}
\newcommand{\Area}{\mathfrak{A}}
\newcommand{\paths}{\mathcal{P}}

\newcommand{\SD}{\text{SD}}

\DeclareMathOperator{\type}{type}

\DeclareMathOperator{\Cat}{Cat}

\DeclareMathOperator{\Id}{Id}
\DeclareMathOperator{\Fix}{Fix}

\author{Valentin F\'eray}
\address{LaBRI, Universit\'e Bordeaux 1, 351 cours de la Lib\'eration, 33 400
Talence, France}
\email{feray@labri.fr}

\title[Complete functions in Jucys-Murphy elements]{On complete functions in Jucys-Murphy elements}


\begin{document}
 
\maketitle

\begin{abstract}
The problem of computing the class expansion of some symmetric functions 
evaluated in Jucys-Murphy elements appears in different contexts, for instance in the computation of matrix integrals.
Recently, M. Lassalle gave a unified algebraic method to obtain
some induction relations on the coefficients in this kind of expansion.
In this paper, we give a simple purely combinatorial proof of his result.
Besides, using the same type of argument, we obtain new simpler formulas.
We also prove an analogous formula in the Hecke algebra of $(S_{2n},H_n)$
and use it to solve a conjecture of S. Matsumoto on the subleading term of orthogonal Weingarten function.
Finally, we propose a conjecture for a continuous interpolation between both problems.
\end{abstract}
 
\section{Introduction}
 
\subsection{Background}
 
The Jucys-Murphy elements $J_{i}$ lie in the symmetric group algebra $\ring[S_{n}]$.
Despite their beautiful properties, their definition is very elementary:
\[J_i = \sum_{j<i} (j\ i)\]
where $(j\ i)$ is the transposition in $S_n$ exchanging $i$ and $j$.
They have been introduced separately by A. Jucys \cite{Jucys1966,Jucys1974} and G. Murphy \cite{Murphy1981}
and have played since a quite important role in representation theory.
Indeed, they act diagonally on the Young basis of any irreducible representation $V_{\lambda}$:
the eigenvalue of $J_{i}$ on an element $e_{T}$ of this basis ($T$ is a standard tableau of shape $\lambda$)
is simply given by the content (\textit{i.e.} the difference between the column-index and the row-index)
of the box of $T$ containing $i$.

In fact, representation theory of symmetric groups $S_n$ can be constructed
entirely using this property (see \cite{OkVe1996}).
We also refer to papers of Biane \cite{BianeAsymptoticsCharacters} and 
Okounkov \cite{Okounkov2000} for nice applications of Jucys-Murphy elements to asymptotic representation theory.\bigskip

A fundamental property, already observed by Jucys and Murphy, is that elementary symmetric functions evaluated in the $J_{i}$'s have a very nice expression (this evaluation is well-defined because Jucys-Murphy elements commute with each other).
More precisely, if $\nbc(\sigma)$ denotes the number of cycles of a permutation $\sigma \in S_n$, then
\begin{equation}\label{EqElJM}
    e_{k}(J_{1},\dots,J_{n}) = \sum_{\sigma \in S_{n} \atop \nbc(\sigma)=n-k} \sigma.
\end{equation}
As this is a central element in the group algebra,
all symmetric functions evaluated in Jucys-Murphy elements are also central.
Therefore it is natural to wonder what their class expansion is.
In other terms, given some symmetric function $F$, can we compute the coefficients $a^F_\lambda$ defined by:
\[ F(J_{1},\dots,J_{n}) =\sum_{\lambda \vdash n} a^F_\lambda \cl_{\lambda},\]
where the sum runs over all partitions of $n$ and 
$\cl_{\lambda}$ denotes the sum of all permutations of cycle-type $\lambda$?
This problem may seem anecdotal at first sight,
but it in fact appears in different domains of mathematics:
\begin{itemize}
  \item When $F$ is a power sum $p_{k}$, it is linked with mathematical physics via vertex operators and Virasoro algebra (see \cite{LascouxThibon2001}).
  \item When $F$ is a complete symmetric function $h_{k}$, the coefficients appearing are exactly the coefficients in the asymptotic expansion of unitary Weingarten functions. The latter is the elementary brick to compute polynomial integrals over the unitary group (see \cite{NovakJMWeingarten, ZinnJustinJMWeingarten}).
  \item The inverse problem (how can we write a given conjugacy class $\cl_{\lambda}$
      as a symmetric function in Jucys-Murphy element) is equivalent
      to try to express character values as a symmetric function of the contents.
      This question has been studied in some papers
      \cite{CorteelGoupilSchaeffer2004, LassalleCaractereExplicite}
      but never using the combinatorics of Jucys-Murphy elements.
\end{itemize}

\subsection{Previous and new results}
As mentioned in the paragraph above, the class expansion of elementary functions in Jucys-Murphy
elements is very simple and was first established by A. Jucys.
The next result of this kind was obtained by A. Lascoux and J.-Y. Thibon via an algebraic method: they gave the coefficients of the class expansion of power sums in Jucys-Murphy elements as some coefficients of an explicit series \cite{LascouxThibon2001}.

Then S. Matsumoto  and J. Novak \cite{MatsumotoNovakMonomialJM} computed the coefficients of the permutations of maximal absolute length in any monomial function in Jucys-Murphy elements.
Their proof is purely combinatorial but does not seem to be extendable to all coefficients.
As monomial functions form a linear basis of symmetric functions, one can deduce from their result a formula for the top coefficients for any symmetric function, in particular for complete functions (see also \cite{MurrayGeneratorsCentreSGA,CollinsMatsumotoOrthogonalWeingarten}).
To be comprehensive, let us add that the authors also obtained all coefficients of cycles in complete symmetric functions using character theory (their approach works for all cycles, not only the ones of maximal length).

Recently, M. Lassalle \cite{LassalleJM} gave a unified method to obtain some induction relations for the coefficients of the class expansion of several families of symmetric functions in Jucys-Murphy elements.
These induction relations allow to compute any coefficient quite quickly.
Besides, it is possible to use them to recover the results of A. Jucys, A. Lascoux and J.-Y. Thibon
and also the top component of complete symmetric functions.
Therefore, the work of M. Lassalle unifies most of the results obtained until now on the subject.

He proves his result with a sophisticated algebraic machinery:
he begins by translating the problem in terms of shifted symmetric functions
and then introduces some relevant differential operators.\medskip

In this paper, we give a simple combinatorial proof of his induction formulas.
Our method of proof can also be adapted to find another formula (Theorem~\ref{ThmNewInd}).
The latter is new and quite simple.
An example of application is the following:
 using Matsumoto's and Novak's result on cycles,
 we are able to compute more generating series of coefficients.
%cross-reference à compléter

\subsection{Generalizations}
An analogous problem can be considered in the Hecke algebra of the Gelfand pair $(S_{2n},H_n)$ 
(here, $H_n$ is the hyperoctahedral group seen as a subgroup of $S_{2n}$).
Definitions are given in section~\ref{SectDoubleClass}.
In this algebra, it is relevant to consider symmetric functions in
odd-indexed Jucys-Murphy elements.

It is a remarkable fact that complete symmetric functions evaluated in these elements are also linked with integrals
over groups of matrices, but the complex unitary group should be replaced
by the real orthogonal group (see \cite{ZinnJustinJMWeingarten,MatsumotoOddJM}).

In paper \cite{MatsumotoOddJM}, S. Matsumoto computed the coefficients of permutations of maximal length
in the case of monomial symmetric functions (hence obtaining an analog of his result with J. Novak).

Our new induction formula extends quite easily to this framework.
A consequence is a proof of a conjecture of S. Matsumoto (see paragraph \ref{SubsubsectProofMatsumoto}). 

In fact, one can even define a generalization of the problem with a parameter $\alpha$
which interpolates between both frameworks:
\begin{itemize}
    \item the class expansion of symmetric functions in Jucys-Murphy elements corresponds to the case $\alpha=1$;
    \item the analog in the Hecke algebra of $(S_{2n},H_n)$ corresponds to the case $\alpha=2$.
\end{itemize}
We recall this construction in section~\ref{SectGeneralisation}.

A very interesting point in Lassalle's method to obtain induction formulas
is that it works almost without changing anything with a general parameter
$\alpha$ \cite[section 11]{LassalleJM}.
Unfortunately, we are not (yet) able to extend our work to this general setting.
However, computer exploration suggests that some of the results still
hold in the general case and 
we present a conjecture in this direction in section~\ref{SectGeneralisation}.
 
\subsection{Organization of the paper}
In section~\ref{SectSymGrpAlg}, we present our results in the symmetric group algebra. 

Then, in section~\ref{SectDoubleClass}, we look at the analogous problem in the Hecke algebra of $(S_{2n},H_n)$.

Finally, in section~\ref{SectGeneralisation}, we present a conjecture for the continuous deformation between these two models.

\section{Induction relations}\label{SectSymGrpAlg}

\subsection{Definitions and notations}
The combinatorics of integer partitions is important in this work
as they index the conjugacy classes in symmetric groups.
A partition $\lambda$ of $n \geq 0$ (we note $\lambda \vdash n$) is a
non-increasing finite sequence of positive integers (called parts) of sum $n$.
Its number of elements is denoted $\ell(\lambda)$.
We use the notation $\lambda \backslash i$ for the partition
obtained from $\lambda$ by erasing one part equal to $i$ (we only use this 
notation when $\lambda$ has at least one part equal to $i$).
In a similar fashion, $\lambda \cup i$ is the partition obtained by adding a part
equal to $i$ (in an appropriate place such that the sequence remains 
non-increasing).

Let us denote by $S_{n}$ the symmetric group of size $n$ and
by $\ring[S_{n}]$ its group algebra over the integer ring.
Throughout the paper, the coefficient of a permutation $\sigma \in S_n$
in an element $x \in \ring[S_n]$ will be denoted $[\sigma] x$.
If this coefficient is non-zero, we say that $\sigma$ is in $x$
(this is a small abuse of language, where we consider $x$ as its support).

\begin{definition}
The Jucys-Murphy elements $J_{i}$ (for $1\leq i \leq n$) are defined by:
\[J_{i} = (1\ i) + (2\ i) + \dots + (i-1\ i)\ \  \in \ring[S_{n}].\]
\end{definition}

Note that $J_{1}=0$ but we include it in our formulas for aesthetic reasons. 

\begin{proposition}
\begin{itemize}
  \item Jucys-Murphy elements commute with each other.
  \item If $F$ is a symmetric function, $F(J_{1},J_{2},\dots,J_{n})$ lies
      in the center of the symmetric group algebra $Z(\ring[S_{n}])$.
\end{itemize}
\end{proposition}

We recall that the cycle-type of a permutation $\sigma$ in $S_{n}$,
which we will denote $\type(\sigma)$,
is by definition the non-increasing sequence of the lengths of its cycles.
This is an integer partition of $n$,
which determines the conjugacy class of the permutation in the group $S_{n}$.

A basis of the center of the group algebra $Z(\ring[S_{n}])$ is given by
the sums of the conjugacy classes, that is the family of elements
\[ \cl_{\lambda} = \sum_{\sigma \in S_{n} \atop \sigma \text{ has cycle-type } \lambda} \sigma, \]
where $\lambda$ runs over all partitions of $n$.
Therefore, for any symmetric function $F$,
there exists some integer numbers $a^F_{\lambda}$ such that:
\[ F(J_{1},\dots,J_{n}) = \sum_{\lambda \vdash n} a^F_{\lambda} \cl_{\lambda}.\]
In other terms, $a^F_{\lambda}$ is the coefficient of any permutation $\sigma$
of type $\lambda$ in $F(J_{1},\dots,J_{n})$.

We will here focus on the case where $F$ is a complete symmetric function
(so $a^{h_{k}}_{\lambda}$ will be denoted $a^{k}_{\lambda}$)
because of the link with some integrals over unitary groups mentioned in the introduction.
Nevertheless, paragraph~\ref{SubsectOtherSym} is devoted to the case of other symmetric functions.

\begin{example}
 As an illustration, let us look at the case $k=2$ and $n=3$:
\begin{align*}
 h_{2}(J_{1},J_{2},J_{3}) &= (1\ 2)^2 + \big((1\ 3) + (2\ 3)\big)^2 + (1\ 2) \cdot \big((1\ 3) + (2\ 3)\big); \\
 &= \Id + 2\Id + (1\ 2\ 3) + (1\ 3\ 2) + (1\ 2\ 3)+(1\ 3\ 2); \\
 &= 3\cl_{1^3} + 2 \cl_{3}.
\end{align*}
Note that the coefficient of a permutation at the end of the computation
does depend only on its cycle-type, although $1$, $2$ and $3$ play different roles in the computation. 

In other terms, we have computed the following coefficients:
\[a^2_{(1^3)} = 3, \quad a^2_{(2\ 1)} = 0, \quad a^2_{(3)} = 2.\]
\end{example}

\subsection{A combinatorial proof of Lassalle's formula}
\label{SubsectLassalleNewProof}
In this paragraph, we give an elementary proof of the following theorem,
which has been proved by M. Lassalle \cite{LassalleJM} using sophisticated algebraic tools.
\begin{theorem}[Lassalle, 2010] \label{ThmLassalle}
For any partition $\rho$ and integer $k$, one has:
\begin{align}
    \label{EqLassalle1} a_{\rho \cup 1}^{k} &= a_\rho^{k} +
    \sum_{i=1}^{\ell(\rho)} \rho_{i}
    a^{k-1}_{\rho \setminus (\rho_{i}) \cup (\rho_{i}+1)} ; \\
    \label{EqLassalle2} \sum_{i=1}^{\ell(\rho)}
    \rho_{i} a^{k}_{\rho \setminus (\rho_{i}) \cup (\rho_{i}+1)} 
    &= \sum_{1\leq i,j \leq \ell(\rho) \atop i \neq j}
    \rho_{i} \rho_{j} a^{k-1}_{\rho \setminus (\rho_{i}, \rho_{j})
    \cup (\rho_{i} + \rho_{j} +1) }  \\
    \nonumber & \qquad\qquad  + \sum_{i=1}^{\ell(\rho)} \rho_i \sum_{r+s=\rho_{i}+1 \atop r,s \geq 1}
a^{k-1}_{\rho \setminus (\rho_{i}) \cup (r,s) }.    \end{align}
\end{theorem}
%We refer to \cite[end of page 13]{LassalleJM} for an explanation of why these equations characterize the numbers $a^k_{\lambda}$ together with initial conditions:
%\begin{align*}
%a_{\rho}^{(0)} &= \left\{ \begin{array}{l}
%1 \text{ if $\rho$ has only parts equal to $1$;} \\
%0 \text{ else;} 
%\end{array} \right.\\
%a_{(1)}^{k} &= \delta_{k,0}.
%\end{align*}
\begin{proof}
We start from the obvious induction relation
\begin{equation}\label{EqRecH}
h_{k}(J_{1},\dots,J_{n+1}) = h_{k}(J_{1},\dots,J_{n}) + J_{n+1} h_{k-1}(J_{1},\dots,J_{n+1}) 
\end{equation}
and we apply to it the following operator:
\[ \E : \begin{array}{rcl}
 \ring[S_{n+1}] & \to & \ring[S_{n}] \\
 \sigma & \mapsto & \left\{
\begin{array}{l}
   \sigma / \{1,\dots,n\} \text{ if }\sigma(n+1) = n+1 ;\\
   0 \text{ else.} 
\end{array}
 \right.
\end{array} \]
Then we look at the coefficient of a permutation $\sigma$ of type $\rho \vdash n$
(in the following, $\sigma'$ is the image of $\sigma$ by the canonical embedding
of $S_{n}$ into $S_{n+1}$, which means that we add $n+1$ as fixed point).
\begin{align}
[\sigma]  \E \big(h_{k}(J_{1},\dots,J_{n+1})\big) &= [\sigma'] h_{k}(J_{1},\dots,J_{n+1}) = a_{\rho \cup 1}^{k}, \label{EqCoefSigmaEHknp1} \\
[\sigma]  \E \big(h_{k}(J_{1},\dots,J_n)\big) &= [\sigma] h_{k}(J_{1},\dots,J_{n}) = a_\rho^{k}, \label{EqCoefSigmaEHkn} \\
[\sigma]  \E \big(J_{n+1} h_{k-1}(J_{1},\dots,J_{n+1})\big)&=
[\sigma'] \sum_{j \leq n} (j \ n+1) h_{k-1}(J_{1},\dots,J_{n+1}) \nonumber\\
& =\sum_{j \leq n} [(j \ n+1) \sigma'] h_{k-1}(J_{1},\dots,J_{n+1}) \nonumber\\
& = \sum_{j \leq n} a^{k-1}_{\type ( (j \ n+1) \sigma' )}, \nonumber
\end{align}
Let us label the cycles of $\sigma$ with the numbers $1,2,\ldots,\ell(\rho)$
such that the $i$-th cycle of $\sigma$ has length $\rho_i$.
It is easy to see that $(j \ n+1) \sigma'$ has exactly the same cycle
decomposition as $\sigma$ except that $n+1$ has been added right before $j$.
Therefore, if $j$ is in the $i$-th cycle of $\sigma$, then $(j \ n+1) \sigma'$
has cycles of length $\rho_1,\rho_2,\dots,\rho_i+1,\dots,\rho_{\ell(\rho)}$.
In other terms, its type is $\rho \setminus (\rho_i) \cup (\rho_i+1)$.
As there are $\rho_i$ elements in the $i$-th cycle of $\sigma$, one obtains:
\begin{equation}
 [\sigma]  \E \big(J_{n+1} h_{k-1}(J_{1},\dots,J_{n+1})\big) =
 \sum_{1 \leq i \leq \ell(\rho)} \rho_i a^{k-1}_{\rho \setminus (\rho_i) \cup (\rho_i+1)}.
    \label{EqCoefSigmaEJHk}
\end{equation}
Putting together equations~\eqref{EqRecH},~\eqref{EqCoefSigmaEHknp1},
\eqref{EqCoefSigmaEHkn} and~\eqref{EqCoefSigmaEJHk}, we obtain the first
part of the theorem.\medskip

The second equality is obtained the same way except that we multiply equation~\eqref{EqRecH} by $J_{n+1}$ before applying the operator $\E$. One obtains:
\begin{multline}\label{EqEJnp1TRecH}
\E \big(J_{n+1} h_{k}(J_{1},\dots,J_{n+1}) \big) = \E \big( J_{n+1} h_{k}(J_{1},\dots,J_{n}) \big) \\
+ \E \big(J_{n+1}^{2} h_{k-1}(J_{1},\dots,J_{n+1}) \big).
\end{multline}
The coefficient of $\sigma$ in the left-hand side has been computed: see equation
\eqref{EqCoefSigmaEJHk}.
Let $\tau$ be a permutation in $h_{k}(J_{1},\dots,J_{n})$. It fixes $n+1$
and, hence, $(j\ n+1) \tau$ can not fix $n+1$ for $j=1,\dots,n$. Therefore,
\begin{equation}\label{EqEJnp1Hkn}
\E \big( J_{n+1} h_{k}(J_{1},\dots,J_{n}) \big) = 0.
\end{equation}
For the last term, we write:
\begin{multline*}\label{EqEJ2H}
[\sigma] \E \big(J_{n+1}^{2} h_{k-1}(J_{1},\dots,J_{n+1}) \big) \\
= [\sigma'] \sum_{j_{1},j_{2} \leq n} (j_{1} \ n+1)\cdot (j_{2}\ n+1) \cdot
h_{k-1}(J_{1},\dots,J_{n+1}) \\
= \sum_{j_{1},j_{2} \leq n}  [(j_{2} \ n+1)\cdot (j_{1}\ n+1) \cdot \sigma']
h_{k-1}(J_{1},\dots,J_{n+1}) \\
= \sum_{j_{1},j_{2} \leq n} a^{k-1}_{\type (
( j_{2} \ n+1)\cdot (j_{1}\ n+1) \cdot \sigma')}.
\end{multline*}
As before, we label the cycles of $\sigma$.
We split the sum in two parts, depending on whether $j_{1}$ and $j_{2}$
are in the same cycle of $\sigma$ or not:
\begin{itemize}
 \item Suppose that both $j_{1}$ and $j_{2}$ are in the $i$-th cycle of $\sigma$.
 That implies that $j_{2} = \sigma^{m}(j_{1})$ for some integer $m$
 between $1$ and $\rho_{i}$
 (eventually $j_1=j_2$, which corresponds to $m=\rho_i$).
 Then $(j_{1} \ n+1)\cdot (j_{2}\ n+1) \cdot \sigma'$ has the same cycles
 as $\sigma$ except for its $i$-th cycle, as well as two other cycles:
  \[ \big(j_{1}, \sigma(j_{1}), \dots \sigma^{m-1}(j_{1}) \big) \text{ and }   \big(j_{2}, \sigma(j_{2}), \dots \sigma^{\rho_{i}-m-1}(j_{2}), n+1 \big).\]
  Thus it has cycle-type $\rho \setminus (\rho_{i}) \cup (m,\rho_{i}-m+1)$.
  There are $\rho_i$ elements in the $i$-th cycle of $\sigma$ and, hence, 
  $\rho_i$ possible values for $j_1$.
  For each value of $j_1$, there is exactly one value of $j_2$ corresponding
  to each value of $m$ between $1$ and $\rho_i$.
  Therefore, one has:
  \begin{align*}
       \sum_{j_{1},j_{2} \leq n \atop j_1 \sim_\sigma j_2}
  a^{k-1}_{\type (( j_{2} \ n+1)\cdot (j_{1}\ n+1) \cdot \sigma')}
  &= \sum_{i \leq \ell(\rho)} \rho_i \sum_{m =1}^{\rho_i}
  a^{k-1}_{\rho \setminus (\rho_{i}) \cup (m,\rho_{i}-m+1)} \\
  &= \sum_{i \leq \ell(\rho)} \rho_i \sum_{r+s = \rho_i+1 \atop r,s \geq 1}
  a^{k-1}_{\rho \setminus (\rho_{i}) \cup (r,s)},
  \end{align*}
  where $j_1 \sim_\sigma j_2$ means that $j_1$ and $j_2$ are in the same cycle
  of $\sigma$.
  \item Let us suppose now that $j_{1}$ and $j_{2}$ are respectively in the
      $i_1$-th and $i_2$-th cycles of $\sigma$ with $i_1 \neq i_2$.
      In this case $(j_{1} \ n+1)\cdot (j_{2}\ n+1) \cdot \sigma'$ has the same
      cycles as $\sigma$ except for its $i_1$-th and $i_2$-th cycles,
      as well as one new cycle:
  \[ \big(j_{1}, \sigma(j_{1}), \dots \sigma^{\rho_{i_{1}}-1}(j_{1}), n+1,
  j_{2}, \sigma(j_{2}), \dots \sigma^{\rho_{i_{2}}-1}(j_{2}) \big). \]
Thus $(j_{1} \ n+1)\cdot (j_{2}\ n+1) \cdot \sigma'$ has cycle-type
$\rho \setminus (\rho_{i_{1}}, \rho_{i_{2}}) \cup (\rho_{i_{1}}+\rho_{i_{2}} + 1)$.
As there are $\rho_{i_1}$ (resp. $\rho_{i_2}$) elements in the $i_1$-th 
(resp. $i_2$-th) cycle of $\sigma$, one obtains:
  \begin{align*}                                                                 
     \sum_{j_{1},j_{2} \leq n \atop j_1 \nsim_\sigma j_2}                  
     a^{k-1}_{\type ( ( j_{2} \ n+1)\cdot (j_{1}\ n+1) \cdot \sigma')}   
     &= \sum_{i_1,i_2 \leq \ell(\rho) \atop i_1 \neq i_2} \rho_{i_1} \rho_{i_2} 
       a^{k-1}_{\rho \setminus (\rho_{i_{1}}, \rho_{i_{2}}) \cup (\rho_{i_{1}}+\rho_{i_{2}} + 1)}, 
     \end{align*} 
     where $j_1 \nsim_\sigma j_2$ means that $j_1$ and $j_2$ are not
     in the same cycle of $\sigma$.
\end{itemize}
Finally,
\begin{multline}\label{EqEJ2H}
[\sigma] \E \big(J_{n+1}^{2} h_{k-1}(J_{1},\dots,J_{n+1}) \big) = \\
\sum_{i \leq \ell(\rho)} \rho_i \sum_{r+s = \rho_i+1 \atop r,s \geq 1}
  a^{k-1}_{\rho \setminus (\rho_{i}) \cup (r,s)}
  + \sum_{i_1,i_2 \leq \ell(\rho) \atop i_1 \neq i_2} \rho_{i_1} \rho_{i_2}         
         a^{k-1}_{\rho \setminus (\rho_{i_{1}}, \rho_{i_{2}}) \cup (\rho_{i_{1}}+\rho_{i_{2}} + 1)}.
\end{multline}
Putting together equations~\eqref{EqEJnp1TRecH},~\eqref{EqCoefSigmaEJHk},
\eqref{EqEJnp1Hkn} and~\eqref{EqEJ2H}, we obtain the second part of the theorem.
\end{proof}

%\begin{remark}
%The arguments in this proof can be used to give an alternative proof of
%central identities in Lassalle's paper.
% More details are given in appendix.
%\end{remark}
\begin{remark}
    This theorem allows to compute inductively the coefficients $a_\rho^k$,
    see \cite[end of page 13]{LassalleJM}.    
\end{remark}

\subsection{New relations}\label{SubsectNewInd}
In this paragraph, we prove new induction relations on the coefficients
$a^k_\rho$, using the same kind of method as above.
\begin{theorem}\label{ThmNewInd}
For any partition $\rho$ and positive integers $k,m$ one has:
\begin{equation}\label{EqNewInd}
 a^{k}_{\rho \cup (m)} = \delta_{m,1} a^k_{\rho} 
 +\sum_{1 \leq i \leq \ell(\rho)} \rho_i
 a^{k-1}_{\rho \setminus (\rho_{i}) \cup (\rho_{i}+m)} 
 + \sum_{r+s=m \atop r,s \geq 1}   a^{k-1}_{\rho \cup (r,s)}.
\end{equation}
\end{theorem}
\begin{proof}
    The case $m=1$ corresponds to equation~\eqref{EqLassalle1} and has already
    been proved.\medskip

    Suppose $m>1$.
    Once again, we begin with equation~\eqref{EqRecH} and
    we will look at the coefficient of some permutation $\sigma$ on both sides.
    
    Let $n=|\rho|+m-1$ and $\sigma$ be a permutation in $S_{n+1}$ of type 
    $\rho \cup (m)$ such that $n+1$ is in a cycle of $\sigma$ of length $m$
    (in particular, as $m>1$, $n+1$ is not a fixed point of $\sigma$).
    By definition,
    \begin{equation}
        [\sigma] h_{k}(J_{1},\dots,J_{n+1}) = a^{k}_{\rho \cup (m)}.
        \label{EqCoef231}
    \end{equation}
    Besides, as all permutations in $h_{k}(J_{1},\dots,J_{n})$ fix $n+1$, but 
    not $\sigma$, one has:
    \begin{equation}
        [\sigma] h_{k}(J_{1},\dots,J_{n}) = 0.
        \label{EqCoef232}
    \end{equation}
    We shall now compute
    \begin{align*}
        [\sigma] J_{n+1} h_{k-1}(J_{1},\dots,J_{n+1}) &=
        [\sigma] \sum_{j \leq n} (j\ n+1) h_{k-1}(J_{1},\dots,J_{n+1}) \\
        &= \sum_{j \leq n} [(j\ n+1) \sigma] h_{k-1}(J_{1},\dots,J_{n+1}) \\
        &= \sum_{j \leq n} a^{k-1}_{\type( (j\ n+1) \sigma)}.        
    \end{align*}
    As before, we label the cycles of $\sigma$: the cycle containing $n+1$ gets
    the label $\ell(\rho)+1$, the others are labelled such that the $i$-th cycle
    has length $\rho_i$ (for $1\leq i\leq \ell(\rho)$).
    We distinguish two cases:
    \begin{itemize}
        \item Suppose that $j$ is in the $\ell(\rho)+1$-th cycle of $\sigma$
            (as $n+1$).
            This implies that $j=\sigma^h(n+1)$ for some $h$ between $1$ and
            $m-1$ (as $j \neq n+1$, $h$ can not be equal to $m$).
            Then $(j\ n+1) \sigma$ has the same cycles than $\sigma$ except
            for its $\ell(\rho)+1$-th cycle, as well as two new cycles:
            \[ \big( n+1, \sigma(n+1), \dots, \sigma^{h-1}(n+1) \big)
            \text{ and } \big(j, \sigma(j), \dots, \sigma^{m-h-1}(j) \big).\]
            Thus its cycle-type is ${\rho \cup (h,m-h)}$.
            Exactly one value of $j$ corresponds to each integer $h$
            between $1$ and $m-1$. One has:
            \[
      \sum_{j \leq n \atop j \sim_\sigma n+1} a^{k-1}_{\type( (j\ n+1) \sigma)}
       = \sum_{h=1}^{m-1} a^{k-1}_{\rho \cup (h,m-h)} 
       =  \sum_{r+s=m \atop r,s \geq 1}   a^{k-1}_{\rho \cup (r,s)}.
           \]
        \item Otherwise, $j$ is in the $i$-th cycle of $\sigma$ for some 
            $i \leq \ell(\rho)$ (in particular, it is not in the same cycle
            as $n+1$).
            In this case, $(j\ n+1) \sigma$ has the same cycles than $\sigma$
            except for its $i$-th and $\ell(\rho)+1$-th cycles,
            as well as one new cycle:
            \[ \big( j, \sigma(j), \dots, \sigma^{\rho_i-1}(j),
                n+1, \sigma(n+1),\dots, \sigma^{m-1}(n+1) \big). \]
            Thus its cycle-type is $\rho \setminus (\rho_{i}) \cup (\rho_{i}+m)$.
            As there are $\rho_i$ elements in the $i$-th cycle of $\sigma$ for each
            $i$, one obtains:
            \begin{align*}   
      \sum_{j \leq n \atop j \nsim_\sigma n+1} a^{k-1}_{\type( (j\ n+1) \sigma)}
      &= \sum_{1 \leq i \leq \ell(\rho)} \rho_i
       a^{k-1}_{\rho \setminus (\rho_{i}) \cup (\rho_{i}+m)}.
   \end{align*}
    \end{itemize}
   Finally,
   \begin{equation}
       [\sigma] J_{n+1} h_{k-1}(J_{1},\dots,J_{n+1}) =
       \sum_{r+s=m \atop r,s \geq 1}   a^{k-1}_{\rho \cup (r,s)} \\
       + \sum_{1 \leq i \leq \ell(\rho)} \rho_i                   
              a^{k-1}_{\rho \setminus (\rho_{i}) \cup (\rho_{i}+m)}.
       \label{EqCoef233}
   \end{equation}
   The theorem follows from equations~\eqref{EqRecH},~\eqref{EqCoef231},
  ~\eqref{EqCoef232} and~\eqref{EqCoef233}.
\end{proof}

\begin{remark}
    This type of case distinctions, depending on whether some elements are in
    the same cycle or not, is quite classical and leads often to
    the same kind of induction relations, called \emph{cut-and-join} equations:
    see for instance \cite{GouldenJacksonCutAndJoin}.
\end{remark}
\begin{remark}
    This theorem implies Theorem~\ref{ThmLassalle}.
    Indeed, equation~\eqref{EqLassalle2} can be written as a linear combination
    of specializations of equation~\eqref{EqNewInd}, but the converse is not true.
\end{remark}
\begin{remark}
    Our new induction relation allows to compute $a_\rho^k$ by induction
    over $|\rho|$ and $k$ in several different ways.
    Indeed, a given partition $\lambda$ can be written as $\rho \cup (m)$ in
    several different ways.
    It is not evident {\it a priori} that the final result does not depend on
    this choice.
    This relies on the initial conditions:
    \[a^1_\rho=\begin{cases}
        1 \text{ if }\rho=2 1^i\text{ for some $i$;}\\
        0 \text{ else.}
    \end{cases} \]
\end{remark}

\subsection{Taking care of the dependence in $n$}
As mentioned by Lassalle \cite[paragraph 2.7]{LassalleJM},
the coefficients $a^k_{\rho \cup 1^{n- |\rho|}}$, seen as functions of $n$, have a very nice structure.
More precisely, let us define $c^{k}_{\lambda}$, where $\lambda$ is a partition,
by induction on $|\lambda|$ by the formula:
\begin{equation}\label{EqLinkAC}
 a^k_{\rho} = \sum_{i = 0}^{m_{1}(\rho)} c^k_{\bar{\rho} \cup 1^i} \binom{m_{1}(\rho)}{i},
 \end{equation}
where $m_{1}(\rho)$ is the number of parts equal to $1$ in $\rho$ and $\bar{\rho}$ is obtained from $\rho$ by erasing its parts equal to $1$.

The interesting fact now is that $c^k_{\rho}$ is equal to $0$
as soon as $|\rho| - \ell(\rho) + m_{1}(\rho)$ is bigger than $k$,
while, for a given $k$, one has infinitely many non-zero $a^k_{\rho}$
(this fact is explained in paragraph \ref{SubsectPartialJM}).
As a consequence, coefficients $c$ are convenient to compute simultaneously
the class expansion of $h_k(J_1,\dots,J_n)$ for all positive integers $n$
(the integer $k$ being fixed):
see Example~\ref{ExCalculC} at the end of this paragraph.

Using equation~\eqref{EqLinkAC}, one can translate Theorems~\ref{ThmLassalle} and~\ref{ThmNewInd} into relations over the $c$'s, but it is rather technical (see \cite[section 12]{LassalleJM}).
We prefer here to explain the combinatorial meaning of the $c$'s and 
derive directly relations over the $c$'s using this interpretation.

\subsubsection{Algebra of partial permutations}
A good tool for that are the partial permutations introduced by Ivanov and Kerov in \cite{IvanovKerovPartialPermutations}.
Let $\B_{\infty}$ be the following $\ring$-algebra:
\begin{itemize}
  \item  A partial permutation is a couple $(d,\sigma)$
      where $d$ is a finite set of positive integers and
      $\sigma$ a permutation of $d$.
    As a $\ring$-module, $\B_{\infty}$ is the set of infinite linear
      combinations of partial permutations.
  \item the product on partial permutations is given by:
      \begin{equation}\label{EqProdPartialPerm}
          (d,\sigma) \cdot (d',\sigma') = 
          (d \cup d', \tilde{\sigma} \cdot \tilde{\sigma'}), \end{equation}
  where $\tilde{\sigma}$ (resp. $\tilde{\sigma'}$) is the canonical
  continuation of $\sigma$ (resp. $\sigma'$) to $d \cup d'$
  (\textit{i.e.} we add fixed points,
  we will use this notation throughout the paper).
  It extends to $\B_{\infty}$ by biliearity:
  \[\left( \sum_{(d,\sigma)} c_{d,\sigma} (d,\sigma) \right) \cdot
  \left( \sum_{(d',\sigma')} c_{d',\sigma'} (d',\sigma') \right)
  = \! \sum_{(d,\sigma), (d',\sigma')} \! c_{d,\sigma} c_{d',\sigma'}
    (d,\sigma) \cdot (d',\sigma').\]
    It is easy to see that in the formula above, only a finite number of term 
    can contribute to the coefficient of a given partial permutation 
    $(d'',\sigma'')$ (indeed, the indices of such terms must fulfill
    $d,d' \subset d''$).
    Therefore the right-hand side is a well-defined element of $\B_\infty$.
\end{itemize}
The infinite symmetric group $S_{\infty}$ acts naturally on $\B_{\infty}$: if $\tau$ belong to $S_{\infty}$, that is $\tau$ is a permutation of $\N^{\star}$ with finite support, we define
\[ \tau \bullet (d,\sigma) = (\tau(d), \tau \sigma \tau^{-1}).\]
The invariants by the action of $S_{\infty}$ form a subalgebra $\A_{\infty}$ of $\B_{\infty}$.
As explained in \cite[§ 6]{IvanovKerovPartialPermutations}, a basis of this subalgebra is
\[ \big(\pcl_{\lambda}\big)_{\lambda \text{ partition}} \text{ where }\pcl_{\lambda}= \!\! \sum_{d \subset \N^\star, \ |d| = |\lambda| \atop \sigma \in S_{d},\ \text{ cycle-type}(\sigma)=\lambda} \!\! (d,\sigma).\]
The nice property of this construction is that, for each $n$,
there exists a morphism $\varphi_{n}$ from $\B_{\infty}$ to 
the symmetric group algebra $\ring[S_{n}]$ defined by:
\[ \varphi_{n}(d,\sigma) = \left\{
\begin{array}{l}
 \tilde{\sigma} \text{ if }d \subset \{1,\dots,n\} ;\\
 0 \text{ else.}
\end{array} \right. \]
These morphisms restrict to morphisms $\A_{\infty} \to Z(\ring[S_{n}])$.
The image of vectors of the basis is given by \cite[equation (4.3)]{IvanovKerovPartialPermutations}:
\[ \varphi_{n}(\pcl_{\lambda}) =
\binom{n -|\lambda|+m_{1}(\lambda)}{m_{1}(\lambda)} 
\cl_{\lambda \cup 1^{n-|\lambda|}}. \]
We shall need a last property of the algebra $\B_\infty$.
Let us define, for a partial permutation $(d,\sigma)$ its degree to be
\[\deg(d,\sigma) = |d| - \#\text{ cycles of } \sigma +
    \#\text{ fixed points of } \sigma. \]
We consider the subspace $(\B_\infty)_{\leq \delta}$ to be the set of
infinite linear combinations of partial permutations of degree smaller
or equal to $\delta$.
\begin{lemma}
    The decomposition $\displaystyle
    \B_\infty= \bigcup_{\delta \geq 1} (\B_\infty)_{\leq \delta}$
    defines an algebra filtration.
\end{lemma}
\begin{remark}
    Consider $\deg'$ defined by 
\[\deg'(d,\sigma) = |d| - \#\text{ cycles of } \sigma,\]
    $\deg'$ is the minimal number of factors needed to write $\sigma$
    (or $\tilde{\sigma}$) as a product of transpositions.
    It is known to define a filtration of $\ring[S_n]$ and hence of $\B_\infty$
    (see \cite[equation (10.3)]{IvanovKerovPartialPermutations}).
\end{remark}
\begin{proof}
    We have to prove that if $(\pi,f)=(\sigma,d) \cdot (\tau,e)$,
    then \[\deg(\pi,f)\leq \deg(\sigma,d) + \deg(\tau,e).\]
    We make an induction on the number $m_1$ of fixed points of $\pi$.

    If $m_1=0$, then
    \[ \deg(\pi,f) = \deg'(\pi,f) \leq \deg'(\sigma,d) + \deg'(\tau,e)
        \leq \deg(\sigma,d) + \deg(\tau,e).\]

    Otherwise, let $i \in f$ be a fixed point of $\pi$.
    We consider the linear operator $F_i$
    \[F_i : \begin{array}{rcl}
        \B_\infty &\longrightarrow & \B_\infty \\
        (\sigma,d) & \longmapsto & \begin{cases}
            (\sigma_{\backslash i},d \backslash \{i\}) \text{ if } i \in d\\
            (\sigma,d) \text{ else,}
        \end{cases}
    \end{array}\]
    where $\sigma_{\backslash i}$ is the permutation obtained by erasing $i$
    in the expression of $\sigma$ as a product of cycles of disjoint supports.
    Equivalenty, by definition, $\sigma_{\backslash i}(j)=\sigma(j)$ if $j \neq \sigma^{-1}(i)$
    and $\sigma_{\backslash i}(\sigma^{-1}(i))=\sigma(i)$.
    It is immediate to check that $\deg(F_i(\sigma,d))=\deg(\sigma,d) - 1$ unless $i$ is in a cycle of 
    length $2$ in $\sigma$, in which case $\deg(F_i(\sigma,d))=\deg(\sigma,d)$.

    Note that $F_i$ is \emph{not} an algebra morphism.
    However, as $i  \in \Fix(\pi)$, one has:
    \begin{equation}\label{EqFiQuasiMorphism}
    F_i(\pi,f)=F_i(\sigma,d) \cdot F_i(\tau,e).
\end{equation}
    Let us explain why this holds. 
    First, it is obvious that 
    \[(d \backslash \{i\}) \cup (e \backslash \{i\})= (d \cup e) \backslash \{i\} = f \backslash \{i\}. \]
    Then, if $j \neq i,\tau^{-1}(i)$, then $\tau_{\backslash i}(j)=\tau(j)$ and thus
    \[\sigma_{\backslash i}\big(\tau_{\backslash i}(j)\big)
    =\sigma_{\backslash i}\big(\tau(j)\big)
    =\sigma\big(\tau(j)\big) = \pi(j)
    = \pi_{\backslash i}(j)\]
    because $\sigma(\tau(j))=\pi(j)$ is different from $i$ (indeed, $\pi(i)=i$ and $j \neq i$).
    Finally, one only has to check that:
    \[\sigma_{\backslash i} \big(\tau_{\backslash i} (\tau^{-1}(i)) \big)
    = \pi_{\backslash i} (\tau^{-1}(i)).\]
    But $\tau_{\backslash i} (\tau^{-1}(i)) =\tau(i)$ and, as 
    $\sigma(\tau(i))=\pi(i)=i$, the left-hand side is equal to
    $\sigma_{\backslash i}(\tau(i))=\sigma(i)$.
    But, as $i$ is a fixed point of $\pi$, the permutation $\pi_{\backslash i}$
    is simply $\pi |_{f \backslash \{i\}}$ and thus the right-hand side is
    equal to $\pi(\tau^{-1}(i))=\sigma(i)$.
    This ends the proof of equation~\eqref{EqFiQuasiMorphism}.

    As $F_i(\pi,f)$ has one less fixed point than $\pi,f$, we can apply the
    induction hypothesis and one has:
    \[\deg(F_i(\pi,f)) \leq \deg(F_i(\sigma,d)) + \deg(F_i(\tau,e)).\]
    As mentioned above:
    \begin{align*}
        \deg(F_i(\pi,f)) &= \deg(\pi,f) +1; \\
        \deg(F_i(\sigma,d)) &= \deg(\sigma,d) +1 -\delta_1; \\
        \deg(F_i(\tau,e)) = \deg(\tau,e) +1 -\delta_2,
    \end{align*}
    where $\delta_1$ (resp. $\delta_2$) is equal to $1$ 
    if $i$ is in a cycle of length $2$ in $\sigma$ (resp. $\tau$)
    and $0$ else.

    If one of the $\delta$'s is equal to $0$, one has
    \begin{multline*}
        \deg\left( (\pi,f) \right) = \deg(F_i(\pi,f)) +1
            \leq \deg(F_i(\sigma,d)) + \deg(F_i(\tau,e)) +1\\
        \leq \deg(\sigma,d) + \deg(\tau,e)
    \end{multline*}
    and the proof is over in this case.
    So the only case we have to study is when $i$ is in cycles of length $2$
    in $\sigma$ and $\tau$.
    Of course, as $\sigma(\tau(i))=i$, both $\sigma(i)$ and $\tau(i)$ are equal
    to the same number $j$.
    In this case, we have:
    \begin{align*}
        F_j(F_i(\pi,f))&=F_j(F_i(\sigma,d)) \cdot F_j(F_i(\tau,e));\\
        \deg(\pi,f)&= \deg\big(F_j(F_i(\pi,f)))+2; \\
        \deg(\sigma,d)&= \deg\big(F_j(F_i(\sigma,d)))+1; \\
        \deg(\tau,e)&= \deg\big(F_j(F_i(\tau,e)))+1
    \end{align*}
    and we can conclude by induction.
%    First denote $d'= d \backslash \Fix(\sigma)$ and $e'=e \backslash
%    \Fix(\tau)$, where $\Fix(\dots)$ is the set of fixed points of the 
%    corresponding permutation.
%    The product $(\sigma,d') \cdot (\tau,e')$ is equal to $(\pi,f')$
%    with $f' = f \backslash (\Fix(\sigma) \cap \Fix(\tau))$. One has:
%    \begin{multline*}
%        \deg(\sigma,d) + \deg(\tau,e) -\deg(\pi,f)
%        = \deg(\sigma/d',d') + |\Fix(\sigma)| + \deg(\tau/e',e') +|\Fix(\tau)
%    \end{multline*}
\end{proof}
\begin{remark}
    In their paper, V. Ivanov and S. Kerov considered a large family of
    filtrations on $\B_\infty$ 
    \cite[Proposition 10.3]{IvanovKerovPartialPermutations},
    but it does not contain this one.
\end{remark}

\subsubsection{Complete functions in partial Jucys-Murphy elements}\label{SubsectPartialJM}
It has been observed in \cite[Section 2]{FerayPartialJM} that,
if we define natural analogs of Jucys-Murphy elements in $\B_{\infty}$ by:
\[X_{i} = \sum_{j < i} \big(\{j,\ i\},\ (j\ i)\big) \quad \text{for }i \geq 1,\]
\begin{itemize}
  \item then they still commute with each other;
  \item besides, the evaluation $F(X_{1}, X_{2}, X_{3}, \dots)$ of any symmetric
      function $F$ in the infinite sequence of partial Jucys-Murphy elements
      is well-defined and lies in $\A_{\infty}$.
\end{itemize}
Therefore there exist coefficients $c^k_{\lambda}$ such that
\[ h_{k}(X_{1},X_{2},X_{3}, \dots) = \sum_{\lambda} c^k_{\lambda} \pcl_{\lambda}. \]
In other terms, $c^k_{\lambda}$ is the coefficient of any partial permutation
$(d,\sigma)$ with $|d|=|\lambda|$ and $\sigma$ of cycle-type $\lambda$ in
$h_{k}(X_{1},X_{2},X_{3}, \dots)$.
Applying $\varphi_{n}$, one obtains:
\begin{align*}
    h_{k}(J_1,\dots,J_n) &= \sum_{\lambda} c^k_{\lambda}
\binom{n -|\lambda|+m_{1}(\lambda)}{m_{1}(\lambda)} 
\cl_{\lambda \cup 1^{n-|\lambda|}} \\
&= \sum_{\rho \vdash n} \left(
\sum_{\lambda \text{ such that} \atop \lambda \cup 1^{n-|\lambda|} = \rho}
c^k_{\lambda} \binom{n -|\lambda|+m_{1}(\lambda)}{m_{1}(\lambda)} \right)
\cl_\rho \\
&= \sum_{\rho \vdash n} \left( \sum_{i=1}^{m_1(\rho)} c^k_{\bar{\rho} \cup 1^i}
\binom{m_{1}(\rho)}{i} \right) \cl_\rho
\end{align*}
Therefore, the numbers $c^k_\lambda$ fulfill equation~\eqref{EqLinkAC} and
this definition is equivalent to the one given at the beginning of the
subsection.
Note that with this construction, it is obvious that the $c$'s are non-negative integers (fact which was observed numerically by Lassalle, private communication).

The fact that $c^k_{\rho}$ is equal to $0$ as soon as
$|\rho| - \ell(\rho) + m_{1}(\rho)$ is bigger than $k$ is also natural
because each $X_i$ is in $(\B_\infty)_{\leq 1}$ and
hence $h_{k}(X_{1},X_{2},X_{3}, \dots)$ lies in $(\B_\infty)_{\leq k}$.
This can of course be generalized to any symmetric function.
In terms of $a$'s, using equation~\eqref{EqLinkAC},
this implies the following property:

\begin{proposition}
    Let $\rho$ be a partition and $F$ a symmetric function of degree $k$.
    The function $t \mapsto a^F_{\rho \cup 1^t}$ is a polynomial in $t$ of degree
    smaller or equal to \[k-(|\rho|-\ell(\rho)).\]
\end{proposition}
The fact that this function is a polynomial is already known 
\cite[Theorem 4.4]{MatsumotoNovakMonomialJM}, but not the bound on the degree.

Besides, we can obtain induction relations on the $c$'s
with the same kind of argument we used for the $a$'s:

\begin{theorem}\label{ThmIndC}
For any partition $\rho$ and positive integers $m$ and $k$, one has
\begin{align*}
 c^k_{\rho \cup 1} &=  \sum_{i} \rho_i c^{k-1}_{\rho \setminus (\rho_i) \cup (\rho_i+1)} ;\\
 c^k_{\rho \cup 2} &=\sum_{i} \rho_i c^{k-1}_{\rho \setminus (\rho_i) \cup (\rho_i+2)}+ c^{k-1}_{\rho \cup (1,1)} + 2 c^{k-1}_{\rho \cup (1)} + c^{k-1}_{\rho} ;\\
 c^k_{\rho \cup m} &=\sum_{i} \rho_i c^{k-1}_{\rho \setminus (\rho_i) \cup (\rho_i+m)}+\sum_{r+s=m \atop r,s \geq 1} c_{\rho \cup (r,s)} + 2 c^{k-1}_{\rho \cup (m-1)} \text{ if }m \geq 3.
\end{align*}
\end{theorem}

\begin{proof}
Let $n+1=|\rho| + m$ and fix a partial permutation $(d,\sigma)$ with:
\begin{itemize}
  \item $d = \{1,\dots,n+1\}$;
  \item $\sigma$ has cycle-type $\rho \cup (m)$ and $n+1$ is in a cycle of length $m$.
\end{itemize}
Let us look at the coefficient $c^k_{\rho \cup m}$ of $(d,\sigma)$ in $h_{k}(X_{1}, X_{2},\dots)$.
As $n+1$ is the biggest element in $d$, it implies that every monomials in the $X_{i}$'s contributing to the coefficient of $(d,\sigma)$ contains no $X_{i}$ with $i>n+1$ and contains at least one $X_{n+1}$. Thus:
\begin{align*}
 c^k_{\rho \cup m} &= [(d,\sigma)] h_{k}(X_{1}, X_{2},\dots) = [(d,\sigma)] X_{n+1} h_{k-1}(X_{1},\dots,X_{n+1}) ; \\ 
 &=  [(d,\sigma)] \sum_{j < n+1}
 \sum_{\nu} \sum_{(d',\tau),\  |d'| = |\nu| \atop \text{cycle-type}(\tau)=\nu}
 c^{k-1}_{\nu} \cdot (d' \cup \{j, n+1\}, (j\ n+1) \tilde{\tau});\\
 &= \sum_{j < n+1}
 \sum_{(d',\tau) \atop \text{cond. }(1)} c^{k-1}_{\type(\tau)},
\end{align*}
where condition $(1)$ is the equality
$(d' \cup \{j, n+1\}, (j\ n+1) \tilde{\tau})=(d,\sigma)$.
For a given integer $j$ between $1$ and $n$, we have to determine which sets
$d'$ and permutations $\tau \in S_{d'}$ fulfill
$d' \cup \{j, n+1\} = d$ and $(j\ n+1) \tilde{\tau} = \sigma$.
Of course, one must have $\tilde{\tau}=(j\ n+1)\sigma$.
As in the previous paragraphs, we make a case distinction:
\begin{itemize}
    \item If $j$ is not in the same cycle of $\sigma$ as $n+1$,
        then they are in the same cycle of $\tilde{\tau}$.
        In particular, neither $j$ nor $n+1$ are fixed points of $\tilde{\tau}$,
        so both belong to $d'$.
        Therefore, necessarily, $d'=d$.
        The discussion on the possible cycle-types of $\tau=\tilde{\tau}$
        is exactly the same than in paragraph~\ref{SubsectNewInd} and one has:
     \[ \sum_{j < n+1 \atop j \nsim n+1} \sum_{(d',\tau) \atop \text{cond. }(1)}
        c^{k-1}_{\type(\tau)} 
      = \sum_{1 \leq i \leq \ell(\rho)} \rho_i
       c^{k-1}_{\rho \setminus (\rho_{i}) \cup (\rho_{i}+m)}.\]
   \item If $j$ is in the same cycle of $\sigma$ as $n+1$ (this implies $m>1$),
       we write $j=\sigma^h(n+1)$.
       If $d'=d$, then $\tau=\tilde{\tau}=(j\ n+1)\sigma$ and its possible
       cycle-types has been discussed in paragraph~\ref{SubsectNewInd}, so one
       has:
       \[ \sum_{j < n+1 \atop j \sim n+1}
       \sum_{(d',\tau) \atop \text{cond. $(1)$ and }d'=d}
       c^{k-1}_{\type(\tau)} 
        =  \sum_{r+s=m \atop r,s \geq 1} c^{k-1}_{\rho \cup (r,s)}.\]
        But, in this case, $d'$ is not necessarily equal to $d$.
        Indeed, when $h=1$, the permutation $\tilde{\tau}$ has $n+1$ as a fixed
        point.
        If $m>2$, $j$ can not be a fixed point in this case so $j \in d'$.
        Therefore $d'=d$ or $d'=d \backslash \{n+1\}$.
      In the last case, $\tau$ is a permutation of cycle-type $\rho \cup (m-1)$.
      A similar phenomenon happens when $h=m-1$: $j$ is a fixed point of
      $\tilde{\tau}$, but not $n+1$, so $d'$ can be equal to 
      $d \backslash \{j\}$ and the corresponding permutation $\tau$ has 
      cycle-type $\rho \cup (m-1)$.
      Therefore, if $m>2$, one has:
      \[ \sum_{j < n+1 \atop j \sim n+1}                        
             \sum_{(d',\tau) \atop \text{cond. $(1)$}}
             c^{k-1}_{\type(\tau)}                          
                =  \sum_{r+s=m \atop r,s \geq 1} c^{k-1}_{\rho \cup (r,s)}
                + 2 c^{k-1}_{\rho \cup (m-1)}.\]
      If $m=2$, the only possible value of $j$ is $\sigma(n+1)$ and, in this
      case, $\tilde{\tau}$ fixes both $j$ and $n+1$.
      Therefore $d'$ can be equal either to $d$, $d \backslash \{n+1\}$,
      $d \backslash \{j\}$ or $d \backslash \{j,n+1\}$.
      It is easy to see that the cycle-types of the corresponding permutations
      $\tau$ are respectively $\rho \cup (1,1)$, $\rho \cup (1)$,
      $\rho \cup (1)$ and $\rho$. Thus, for $m=2$,
      \[ \sum_{j < n+1 \atop j \sim n+1}                         
               \sum_{(d',\tau) \atop \text{cond. $(1)$}}           
                        c^{k-1}_{\type(\tau)}=
     c^{k-1}_{\rho \cup (1,1)} + 2 c^{k-1}_{\rho \cup (1)} + c^{k-1}_{\rho}.\]
\end{itemize}
Summing the different contributions in the different cases,
we obtain our theorem.
\end{proof}

\begin{example}\label{ExCalculC}
Here are the non-zero values of $c^k_\rho$ for small values of $k$ ($k \leq 3$).
It is immediate that $c^1_{(2)}$ is equal to $1$, while all other $c^1_\mu$ are $0$.
Then Theorem~\ref{ThmIndC} allows to compute:
\begin{align*}
c^2_{(1,1)}&= 1\cdot c^1_{(2)}=1 ;\\
c^2_{(2,2)}&= 2 c^1_{(4)} + c^1_{(2,1,1)} + 2 c^1_{(2,1)} + c^1_{(2)} = 1 ;\\
c^2_{(3)} &= 2 c^1_{(2,1)} + 2 c^1_{(2)} = 2;\\
c^3_{(2)} &= c^2_{(1,1)} = 1;\\
c^3_{(2,1)} &= 2 c^2_{(3)} = 4;\\
c^3_{(2,1,1)} &= 2 c^2_{(3,1)} + c^2_{(2,2)} = 1; \\
c^3_{(2,2,2)} &= c^2_{(2,2)}=1; \\
c^3_{(3,2)} & = c^2_{(3)} = 2; \\
c^3_{(4)} & = c^2_{(2,2)} + 2 c^2_{(3)}=5.
\end{align*}
Using equation~\eqref{EqLinkAC}, we can compute all coefficients $a^k_\rho$ for $k=2,3$ and we find the following class expansion (true for any $n\geq 1$):
\begin{align*}
 h_2(J_1,\ldots,J_n) & = \delta_{n\geq 3}\  2 \cl_{(3,1^{n-3})} + \delta_{n \geq 4} \ \cl_{(2,2,1^{n-4})} + \binom{n}{2} \cl_{1^n};\\
 h_3(J_1,\ldots,J_n) & = \delta_{n\geq 4}\  5 \cl_{(4,1^{n-4})} + \delta_{n \geq 5} \ 2 \cl_{(3,2,1^{n-5})} + \delta_{n \geq 6} \ \cl_{(2,2,2,1^{n-6})}\\
& \qquad + \delta_{n \geq 2} \ \left(\binom{n-2}{2} + 4 \binom{n-2}{1} + \binom{n-2}{0}\right) \cl_{2,1^{n-2}}.
\end{align*}
This kind of results could also have been obtained with Theorem~\ref{ThmNewInd}
%(the initial conditions for the $a$'s are the following:
%$a^1_{(2,1^{n-2})}=1$ for any $n \geq 2$ while other values of $a^1_\mu$ are equal to $0$),
but the computation is a little harder (it involves discrete integrals of polynomials).
\end{example}

\subsection{Generating series for some coefficients}
S. Matsumoto and J. Novak have computed, using character theory,
the following generating function \cite[Theorem 6.7]{MatsumotoNovakMonomialJM}.
\begin{theorem}[Matsumoto, Novak, 2009]
 For any integer $n \geq 2$, one has:
\begin{equation}
  \sum_{k} a^{k}_{(n)} z^k = \frac{\Cat_{n-1} z^{n-1}}{(1-1^2 z^2)(1-2^2 z^2)\dots(1-(n-1)^2 z^2)}, 
\end{equation}
 where $\Cat_{n-1} = \frac{1}{n}\binom{2(n-1)}{n-1}$ is the usual Catalan number.
\end{theorem}

As $a^{k}_{(n)}=c^k_{(n)}$, the same result holds on the $c$'s.
Unfortunately, we are not able to find a proof of their formula 
{\it via} Theorem~\ref{ThmIndC},
but the latter can be used to derive new results of the same kind.
\medskip

For instance, with $\rho=(n-1)$ and $m=1$, our induction relation writes as
$c_{(n-1,1)}^k=(n-1) \cdot c_{(n)}^{k-1}$ and thus
\begin{multline*}
 \sum_{k} c^{k}_{(n-1,1)} z^k = z \sum_{k} (n-1) c^{k-1}_{(n)} z^{k-1} \\
  = \frac{(n-1) \Cat_{n-1} z^n}{(1-1^2 z^2)(1-2^2 z^2)\dots(1-(n-1)^2 z^2)}.
\end{multline*}
In terms of $a$'s, this result implies:
\begin{align}
 \nonumber \sum_{k} a^{k}_{(n-1,1)} z^k &= \sum_{k} \left(c^k_{(n-1,1)} + c^k_{(n-1)}\right) z^k \\
 &= \frac{(n-1) \Cat_{n-1} z^{n} + (1-(n-1)^2 z^2) \Cat_{n-2} z^{n-2}}{(1-1^2 z^2)(1-2^2 z^2)\dots(1-(n-1)^2 z^2)}. 
\end{align}
This expression is simpler than the one obtained by Matsumoto and Novak for the same quantity \cite[Proposition 6.9]{MatsumotoNovakMonomialJM} and their equivalence is not obvious at all.\medskip

If we want to go further and compute other generating series, one has to solve linear systems. For instance, denoting $F_{\mu}= \sum_{k} c^k_{\mu} z^k$, Theorem~\ref{ThmIndC} gives:
\begin{align*}
F_{(n-2,1,1)} &= z  \left((n-2) F_{(n-1,1)} + F_{(n-2,2)}\right);\\
F_{(n-2,2)} &= z \left((n-2) F_{(n)} + F_{(n-2,1,1)} + F_{(n-2,1)} + F_{(n-2)} \right) .
\end{align*}
After resolution, one has:
\begin{align*}
F_{(n-2,1,1)} &= \frac{z^2 \left( n (n-2) F_{(n)} + z (n-2) F_{(n-1)} + F_{(n-2)} \right) }{1-z^2} ; \\
F_{(n-2,2)} &= \frac{ z \left( (n-2) F_{(n)} + F_{(n-2,1)} + F_{(n-2)} \right) + z^2 (n-2) F_{(n-1,1)} }{1-z^2}.
\end{align*}
Using the results above, one can deduce an explicit generating series for the $c$'s which can be easily transformed into series for the $a$'s.

\subsection{Other symmetric functions}\label{SubsectOtherSym}
Even if we have focused so far on complete symmetric functions,
our method works also with power-sums.

The induction equation~\eqref{EqRecH} should be replaced by:
\begin{align*}
p_{k}(J_{1},\dots,J_{n+1}) &= p_{k}(J_{1},\dots,J_{n}) + J_{n+1} p_{k-1}(J_{1},\dots,J_{n+1})\\
& \qquad -J_{n+1} p_{k-1}(J_{1},\dots,J_{n}).
\end{align*}
Then, using similar arguments to the ones of paragraph~\ref{SubsectNewInd},
one gets the following induction relation:
\begin{align*}
a^{p_k}_{\rho \cup (m)} &= \delta_{m,1} a^{p_k}_{\rho}
 +\sum_{1 \leq i \leq \ell(\rho)} \rho_i
 a^{p_{k-1}}_{\rho \setminus (\rho_{i}) \cup (\rho_{i}+m)} \\
 &\quad
 + \sum_{r+s=m \atop r,s \geq 1} a^{p_{k-1}}_{\rho \cup (r,s)}
 -\delta_{m>1}\ a^{p_{k-1}}_{\rho \cup (m-1)}.
 \end{align*}

 Unfortunately, we are not able to deal with a linear basis of the symmetric function ring
(as the coefficients $a^F_{\lambda}$ depend linearly on $F$, this would solve
the problem for all symmetric functions).

\section{Analogs in the Hecke algebra of $(S_{2n},H_n)$}
\label{SectDoubleClass}
In this section, we consider a slightly different problem, which happens to be the analog of the one of the previous section.
It was first considered recently by P. Zinn-Justin \cite{ZinnJustinJMWeingarten}
and S. Matsumoto \cite{MatsumotoOddJM} in connection with integrals over
orthogonal groups.

\subsection{Hecke algebra of $(S_{2n},H_n)$}
The results of this section are quite classical. A good survey, with a more representation-theoretical point of view, can be found in I.G. Macdonald's book\cite[Chapter 7]{McDo}.

Let us consider the symmetric group of even size $S_{2n}$,
whose elements are seen as permutations of the set $\{1,\bar{1},\ldots,n,\bar{n}\}$.
It contains the hyperoctahedral group which is the subgroup formed by
permutations $\sigma \in S_{2n}$ such that $\overline{\sigma(i)}=\sigma(\bar{i})$
(by convention, $\overline{\bar{i}}=i$).
We are interested in the double cosets $H_{n} \backslash S_{2n} / H_{n}$, \textit{i.e.} the equivalence classes for the relation:
\[\sigma \equiv \tau \text{ if and only if }\exists\ h,h' \in H_{n} \text{ s.t. } \sigma = h \tau h'.\]

Conjugacy classes in the symmetric group algebra can be characterized
easily using cycle-types.
We recall a similar result for the double cosets: 
they are characterized \textit{via} coset-types.

\begin{definition}\label{DefCosetType}
 Let $\sigma$ be a permutation of $S_{2n}$. Consider the following graph $G_{\sigma}$:
 \begin{itemize}
  \item its $2n$ vertices are labelled by $\{1,\bar{1},\ldots,n,\bar{n}\}$;
  \item we put a solid edge between $i$ and $\bar{i}$ and a dashed one between $\sigma(i)$ and $\sigma(\bar{i})$ for each $i$.
\end{itemize}
Forgetting the types of the edges, we obtain a graph with only vertices of degree $2$. Thus, it is a collection of cycles.
Moreover, due to the bicoloration of edges, it is easy to see that all these cycles have an even length.

We call coset-type of $\sigma$ the partition $\mu$ such that the lengths of the cycles of $G_{\sigma}$ are equal to $2\mu_{1}, 2\mu_{2}, \dots$
\end{definition}

\begin{example}
 Let $n=4$ and $\sigma$ be the following permutation:
 \[1 \mapsto 3,\ \bar{1} \mapsto 1,\ 2 \mapsto \bar{4},\ \bar{2} \mapsto \bar{3},\ 3 \mapsto \bar{2},\ \bar{3} \mapsto 2,\ 4\mapsto 4,\ \bar{4} \mapsto \bar{1}.\]
 The corresponding graph $G_{\sigma}$ is drawn on figure~\ref{FigGSigma}.

\begin{figure}[ht]
\[\includegraphics[width=5cm]{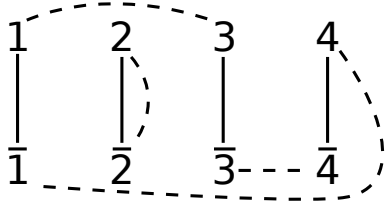}\]
\caption{Example of graph $G_{\sigma}$}
\label{FigGSigma}
\end{figure}
\end{example}

This graph is the disjoint union of one cycle of length $6$ ($1,3,\bar{3},\bar{4},4,\bar{1}$) and one cycle of length $2$ ($2,\bar{2}$).
Thus the coset-type of $\sigma$ is the integer partition $(3,1)$.

\begin{proposition}\cite[section 7.1]{McDo}
 Two permutations are in the same double coset
 if and only if their coset-types are the same.
\end{proposition}

If $\mu$ is a partition of $n$, we denote
\[\cl^{(2)}_{\mu}=\sum_{\sigma \in S_{2n} \atop \text{coset-type}(\sigma)=\mu} \sigma \ \in \ring[S_{2n}].\]
It is immediate that the elements $\cl^{(2)}_{\mu}$, when $\mu$ runs over partitions of $n$ span linearly a subalgebra $Z_n^{(2)}$ of $\ring[S_{2n}]$.
Equivalently, one can define $Z_n^{(2)}$ as the algebra of functions on $S_{2n}$, invariant by left and right multiplication by an element of $H_{n}$, endowed with the convolution product
\[ f \star g(\sigma) = \sum_{\tau_1, \tau_2 \in S_{2n} \atop \tau_{1} \tau_{2}=\sigma}
f(\tau_{1}) g(\tau_{2}).\]
One can prove using representation theory \cite[section 7.2]{McDo} that this algebra is commutative (in other terms, $(S_{2n},H_{n})$ is a Gelfand pair).

\subsection{Odd Jucys-Murphy elements}
In this section we will look at symmetric functions in odd-indexed Jucys-Murphy elements in $S_{2n}$.
Rewriting as permutations on the set $\{1,\bar{1},2,\bar{2},\dots,n,\bar{n}\}$ (ordered by $1<\bar{1}<2<\bar{2}<\dots<n<\bar{n}$), these elements are:
\[J^{(2)}_{i} = \sum_{j=1,\bar{1},\dots,i-1,\overline{i-1}} (j\ i).\]
They were considered by P. Zinn-Justin \cite{ZinnJustinJMWeingarten}
and then S. Matsumoto \cite{MatsumotoOddJM}.

Let us consider also the following element in $\QQ[S_{2n}]$:
\[p_{n}= \sum_{h \in H_{n}} h. \]
Then the following result holds, which may be seen as an analog of the fact
that symmetric functions in Jucys-Murphy elements are central in the symmetric group
algebra.
\begin{proposition}
 If $F$ is a symmetric function, then:
 \[x_{n,F} := F(J^{(2)}_{1}, \dots, J_{n}^{(2)}) p_{n} = p_{n} F(J_{1}^{(2)},\dots,J_{n}^{(2)}). \]
 Moreover $x_{n,F}$ belongs to the algebra $Z_n^{(2)}$.
\end{proposition}
\begin{proof}[Sketch of proof]
The first step is to prove by induction that
\[ e_{k}(J_{1}^{(2)},\dots,J_{n}^{(2)}) p_{n} =  p_{n} e_{k}(J_{1}^{(2)},\dots,J_{n}^{(2)}) = \sum_{\mu \vdash n \atop |\mu| - \ell(\mu) = k} \cl^{(2)}_{\mu}.\]
The result follows for all $F$ by multiplication and linear combination.
See \cite[Proposition 3]{ZinnJustinJMWeingarten} and 
\cite[Proposition 3.1]{MatsumotoOddJM} for details.
\end{proof}

Inspired by the results of section~\ref{SectSymGrpAlg},
we may look at the class expansion of $x_{n,F}$,
\textit{i.e.} the coefficients $b^{F}_{\mu}$ such that:
\[F(J_{1}^{(2)},\dots,J_{n}^{(2)}) p_{n} = \sum_{\mu \vdash n} b^{F}_{\mu} \cl^{(2)}_{\mu}.\]
As seen in the sketch of proof for the proposition above, the $b$'s are easy to compute in the case of elementary functions.

In the following paragraph, we will establish some induction relations for the $b$'s in the case of complete symmetric functions.
We focus on this case (and thus use the short notation $b^k_{\mu}=b^{h_{k}}_{\mu}$) because these coefficients appear in the computation of the asymptotic expansion of some integrals over the orthogonal group \cite[Theorem 7.3]{MatsumotoOddJM}.

\subsection{A simple induction relation}
In this paragraph, using the same method as in subsection~\ref{SubsectNewInd}, we prove the following induction formula for the $b$'s.

\begin{theorem}\label{ThmNewInd2}
For any partition $\rho$ and positive integers $k$ and $m$, one has:
\begin{equation}\label{EqNewInd2}
 b^{k}_{\rho \cup (m)} = \delta_{m,1} b^k_{\rho} +
 2 \sum_{1 \leq i \leq \ell(\rho)} \rho_i
 b^{k-1}_{\rho \setminus (\rho_{i}) \cup (\rho_{i}+m)}
 + \sum_{r+s=m \atop r,s \geq 1} b^{k-1}_{\rho \cup (r,s)} 
 + (m-1) b^{k-1}_{\rho \cup (m)}.
\end{equation}
\end{theorem}
\begin{proof}
 As before, the starting point of our proof is an induction relation on complete symmetric functions:
\begin{multline*}
 h_{k}(J_{1}^{(2)},\dots,J_{n}^{(2)},J_{n+1}^{(2)}) = h_{k}(J_{1}^{(2)},\dots,J_{n}^{(2)})\\
 	 + J^{(2)}_{n+1} h_{k-1}(J_{1}^{(2)},\dots,J_{n}^{(2)},J_{n+1}^{(2)}).
\end{multline*}
Multiplying both sides by $p_{n+1}$, one has:
\begin{multline}\label{EqRecH2}
  h_{k}(J_{1}^{(2)},\dots,J_{n}^{(2)},J_{n+1}^{(2)}) \cdot p_{n+1}=
  h_{k}(J_{1}^{(2)},\dots,J_{n}^{(2)}) \cdot p_{n+1} \\
+ J^{(2)}_{n+1} h_{k-1}(J_{1}^{(2)},\dots,J_{n}^{(2)},J_{n+1}^{(2)}) \cdot p_{n+1}.
\end{multline}

Let us begin with the case $m=1$.
We choose a permutation $\sigma \in S_{2n}$ of coset-type $\rho$ and
we denote $\sigma'$ its image by the canonical embedding
$S_{2n} \hookrightarrow S_{2n+2}$.
It has coset-type $\rho \cup (1)$.
By definition,
\begin{equation}
 [\sigma'] h_{k}(J_{1}^{(2)},\dots,J_{n}^{(2)},J_{n+1}^{(2)}) p_{n+1} =
 b^{k}_{\rho \cup 1} \label{EqCoef331}
\end{equation}
For the second term, we write:
\begin{multline*}
h_{k}(J_{1}^{(2)},\dots,J_{n}^{(2)}) \cdot p_{n+1} =
h_{k}(J_{1}^{(2)},\dots,J_{n}^{(2)}) \cdot p_n \\
\cdot \left(1 + (n+1\ \overline{n+1}) + 
\sum_{i=1,\bar{1},\dots,n,\bar{n}} (n+1\ i) (\overline{n+1}\ \bar{i}) \right).
\end{multline*}
Notice that $h_{k}(J_{1}^{(2)},\dots,J_{n}^{(2)}) p_{n}$ lies in the algebra
$\ring[S_{2n}] \subset \ring[S_{2n+2}] $ and hence
is a linear combination of permutations fixing $n+1$ and $\overline{n+1}$.
For such permutations $\tau$, neither $\tau (n+1\ \overline{n+1})$
nor $\tau (n+1\ i) (\overline{n+1}\ \bar{i})$ can be equal to $\sigma'$
(these two permutations do not fix $n+1$ and $\overline{n+1}$).
Therefore,
\begin{equation}
    [\sigma'] h_{k}(J_{1}^{(2)},\dots,J_{n}^{(2)})\cdot  p_{n+1}
 = [\sigma] h_{k}(J_{1}^{(2)},\dots,J_{n}^{(2)}) \cdot p_{n} = b^k_{\rho}.
 \label{EqCoef332}
\end{equation}

We still have to compute:
\begin{multline}\label{EqTechniqueCoef}
 [\sigma'] J^{(2)}_{n+1} 
 h_{k-1}(J_{1}^{(2)},\dots,J_{n}^{(2)},J_{n+1}^{(2)}) \cdot p_{n+1}  \\
 = \sum_{j = 1,\bar{1},\dots,n,\bar{n}} [(n+1\ j) \sigma']
 h_{k-1}(J_{1}^{(2)},\dots,J_{n}^{(2)},J_{n+1}^{(2)}) \cdot p_{n+1} \\
 = \sum_{j = 1,\bar{1},\dots,n,\bar{n}}
 b^{k-1}_{\text{coset-type}((n+1\ j) \sigma')}.
\end{multline}
Let us look at the coset-type of $(n+1\ j) \sigma'$.
Denote by $d_{j}$ (resp. $d_{n+1}$) the other extremity of the dashed edge 
of extremity $j$ (resp. $n+1$) in $G_{\sigma'}$
(see definition~\ref{DefCosetType}).
Then the graph $G_{(n+1\ j) \sigma'}$ has exactly the same edges as
$G_{\sigma'}$, except for $(j,d_{j})$ and $(n+1,d_{n+1})$,
which are replaced by $(j,d_{n+1})$ and $(n+1,d_{j})$.

As $(n+1,\overline{n+1})$ is a loop of length $2$ in $G_{\sigma'}$,
if we assume that $j$ was in a loop of size $2 \rho_{i}$,
then these two loops are replaced by a loop of size $2 \rho_{i} + 2$
in $G_{(n+1\ j) \sigma'}$ (it is a particular case of the phenomenon
drawn on Figure~\ref{FigJoinLoop}).

  \begin{figure}[ht]
\[\begin{array}{c}
\includegraphics[width=5cm]{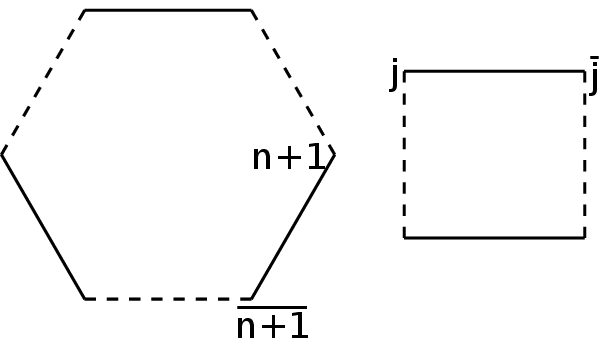}
\end{array}
\rightarrow
\begin{array}{c}
 \includegraphics[width=5cm]{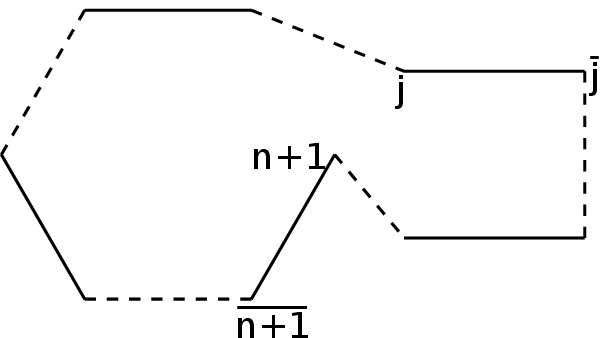}
\end{array}
\]
\caption{$G_{\sigma}$ and $G_{(n+1\ j) \sigma}$ in the ``join'' case}
\label{FigJoinLoop}
\end{figure}
Therefore $(n+1\ j) \sigma'$ has coset-type
$\rho \backslash  \rho_{i} \cup (\rho_{i} +1)$.
As there are $2\rho_i$ elements in the $i$-th loop of $G_\sigma'$, one obtains:
\begin{equation}
    \label{EqCoef333}
[\sigma'] J^{(2)}_{n+1} h_{k-1}(J_{1}^{(2)},\dots,J_{n}^{(2)},J_{n+1}^{(2)})
\cdot p_{n+1} = 2 \sum_{1 \leq i \leq \ell(\rho)}
\rho_i b^{k-1}_{\rho \setminus (\rho_{i}) \cup (\rho_{i}+1)}
\end{equation}
Putting together equations \eqref{EqRecH2}, \eqref{EqCoef331}, \eqref{EqCoef332}
and \eqref{EqCoef333}, we obtain the case $m=1$ of the theorem.\medskip

Let us consider now the case $m>1$.
We choose a permutation $\sigma \in S_{2n+2}$ of coset-type $\rho \cup (m)$
such that $n+1$ is in a loop of size $2m$ in $G_{\sigma}$.
As $m>1$, this implies that
$\overline{\sigma^{-1}(n+1)} \neq \sigma^{-1}(\overline{n+1})$.
On the other hand, if $\tau$ lies in $\ring[S_{2n}] \subset \ring[S_{2n+2}]$
and $i=1,\bar{1},\dots,n,\bar{n}$, one has:
\begin{align*}
    \big( \tau (i\ n+1) (\overline{i}\ \overline{n+1}) \big)^{-1} (n+1) &= i ;\\
    \big( \tau (i\ n+1) (\overline{i}\ \overline{n+1}) \big)^{-1} (\overline{n+1}) &= \overline{i}.
\end{align*}
Thus, $\sigma$ can not be written as $\tau (i\ n+1) (\overline{i}\ \overline{n+1})$
with the conditions above.
It can not be equal to $\tau$ or written as $\tau (n+1\ \overline{n+1})$ either.
Therefore,
\[ [\sigma] h_{k}(J_{1}^{(2)},\dots,J_{n}^{(2)}) p_{n+1} = 0 \]
As a consequence, one has:
\begin{multline*}
b^k_{\rho \cup (m)} =
[\sigma] h_{k}(J_{1}^{(2)},\dots,J_{n}^{(2)},J_{n+1}^{(2)})\cdot p_{n+1}\\
= [\sigma] J^{(2)}_{n+1} h_{k-1}(J_{1}^{(2)},\dots,J_{n}^{(2)},J_{n+1}^{(2)})
\cdot p_{n+1} \\
=  \sum_{j = 1,\bar{1},\dots,n,\bar{n}}
b^{k-1}_{\text{coset-type}( ( n+1\ j) \sigma')}.
\end{multline*}
and we have to look at the possible coset types of $(n+1\ i) \sigma$ (equation~\eqref{EqTechniqueCoef} is still true).

Let us number the loops of the graph $G_\sigma$ with the integers 
$1,2,\dots,\ell(\rho)+1$ such that the $i$-th loop has length $2 \rho_i$
for $i \leq \ell(\rho)$ and the $\ell(\rho)+1$-th loop is the one containing
$n+1$.
As before, the graph $G_{(n+1\ j) \sigma}$ is obtained from $G_{\sigma}$ by
replacing edges $(j,d_{j})$ and $(n+1,d_{n+1})$ by $(j,d_{n+1})$ and
$(n+1,d_{j})$.
We distinguish three cases:

\begin{description}
    \item[``join''] If $j$ lies in the $i$-th loop of $G_{\sigma}$, then 
        $G_{(n+1\ j) \sigma}$ is obtained from $G_{\sigma}$ by erasing
        its $i$-th and $\ell(\rho)+1$-th loops and replacing them
        by a loop of size $2 (\rho_{i} +m)$ (see figure~\ref{FigJoinLoop}).
  In this case, $(n+1\ i) \sigma$ has coset-type
  $\rho \backslash  \rho_{j} \cup (\rho_{j} +m)$.

  As there are $2 \rho_i$ elements in the $i$-th loop of $G_\sigma$, one obtains:
  \[\sum_{j=1,\bar{1},\dots,n,\bar{n} \atop j \nsim_{G_\sigma} n+1}
  b^{k-1}_{\text{coset-type}( ( n+1\ i) \sigma')}
  = 2 \sum_{1 \leq i \leq \ell(\rho)} \rho_i
  b^{k-1}_{\rho \setminus (\rho_{i}) \cup (\rho_{i}+m)}, \]
  where $j \nsim_{G_\sigma} n+1$ means that $j$ and $n+1$ lie in different
  loops of $G_\sigma$.

  \item[``twist''] If $j$ lies in the $\ell(\rho)+1$-th loop of $G_{\sigma}$
      and if the distance between $j$ and $n+1$ is odd,
      then $G_{(n+1\ i) \sigma}$ is obtained from $G_{\sigma}$ by
      the transformation drawn in figure~\ref{FigTwist}.
      In particular, in this case, $(n+1\ j) \sigma$ has the same
      coset-type as $\sigma$, that is $\rho \cup (m)$.
      
      As $j$ can not be equal to $\overline{n+1}$, there are $m-1$ possible
      values for $j$ in this case. Thus,
      \[\sum_{j=1,\bar{1},\dots,n,\bar{n} \atop {j \sim_{G_\sigma} n+1
      \atop d_G(j,n+1) \text{ odd}}} 
      b^{k-1}_{\text{coset-type}( ( n+1\ j) \sigma')}
      = (m-1) b^{k-1}_{\rho \cup (m)}.\]

    \begin{figure}[ht]
\[\begin{array}{c}
\includegraphics[width=3.5cm]{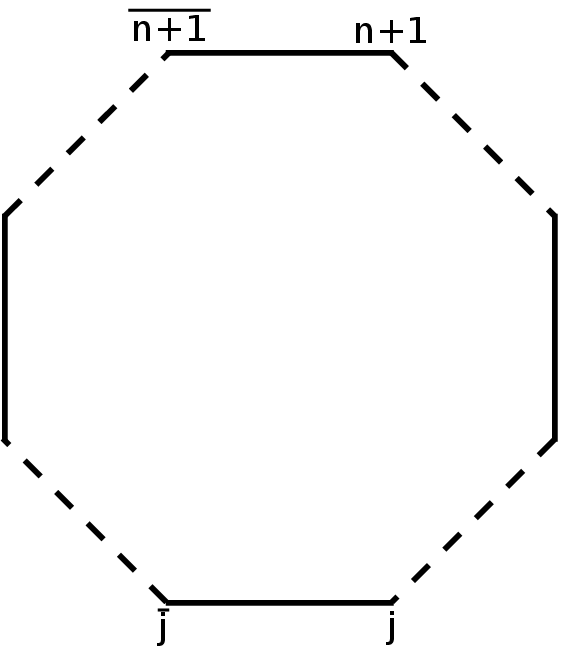}
\end{array}
\rightarrow
\begin{array}{c}
 \includegraphics[width=3.5cm]{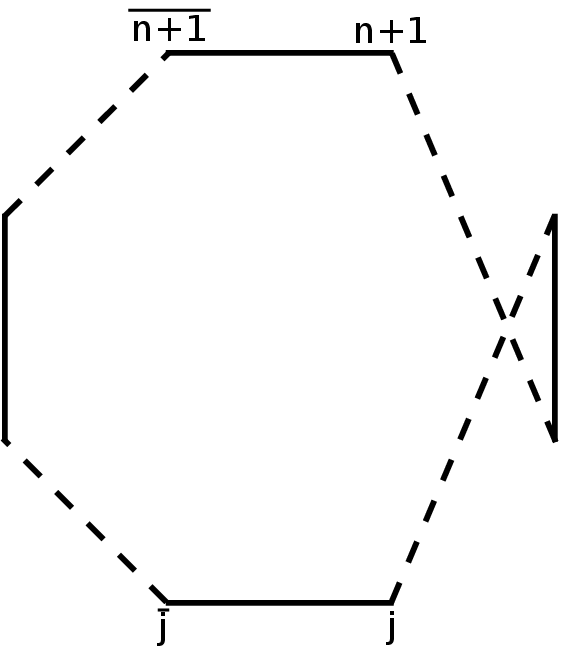}
\end{array}
\]
\caption{$G_{\sigma}$ and $G_{(n+1\ j) \sigma}$ in the ``twist'' case}
\label{FigTwist}
\end{figure}

  \item[``cut''] We consider now the case where $j$ lies in the $\ell(\rho)+1$-th loop
      of $G_{\sigma}$ and the distance between $j$ and $n+1$ is even.
      We choose an arbitrary orientation of the $\ell(\rho)+1$-th loop of
      $G_{\sigma}$ (we keep the same for all $j$ in this situation) and
      we denote $2h$ ($1\leq h \leq m-1$) the distance between $n+1$ and $j$
      when following the loop along this direction.
      Then $G_{(n+1\ j) \sigma}$ is obtained from $G_{\sigma}$ by erasing
      its $\ell(\rho)+1$-th loop and replacing it by two loops of length $2h$
      and $2(m-h)$.(see figure~\ref{FigCutLoop}).
      Thus, in this case, $(n+1\ j) \sigma$ has coset-type $\rho \cup (h,m-h)$.
  
      There is exactly one integer $j$ for each integer $h$ between $1$ and
      $m-1$, so:
      \begin{multline*}
          \sum_{j=1,\bar{1},\dots,n,\bar{n} \atop {j \sim_{G_\sigma} n+1
          \atop d_G(j,n+1) \text{ even}} }
          b^{k-1}_{\text{coset-type}( ( n+1\ j) \sigma')}
          = \sum_{h=1}^{m-1} b^{k-1}_{\rho \cup (h,m-h)} 
          = \sum_{r+s=m \atop r,s \geq 1} b^{k-1}_{\rho \cup (r,s)}
      \end{multline*}

  \begin{figure}[ht]
\[\begin{array}{c}
\includegraphics[width=3.5cm]{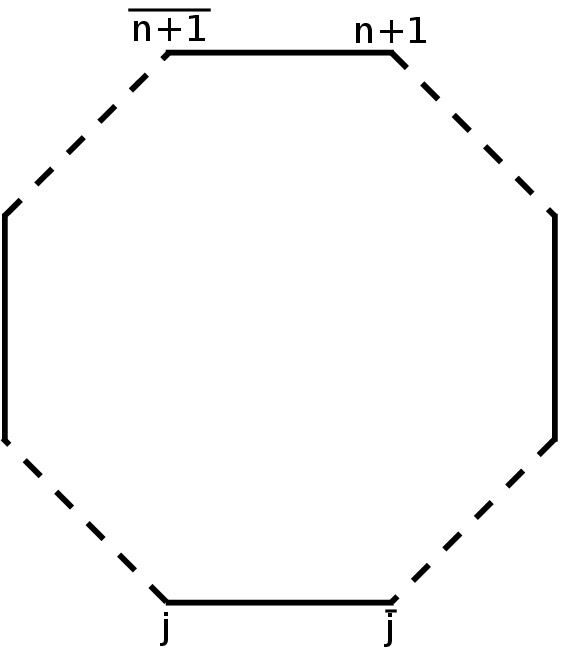}
\end{array}
\rightarrow
\begin{array}{c}
 \includegraphics[width=3.5cm]{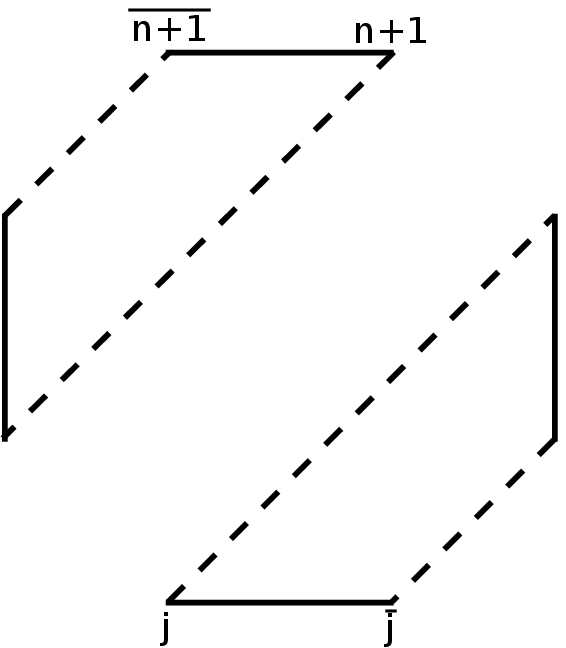}
\end{array}
\]
\caption{$G_{\sigma}$ and $G_{(n+1\ j) \sigma}$ in the ``cut'' case}
\label{FigCutLoop}
\end{figure}
\end{description}
Putting the different cases together, one has
 \begin{multline*}
 b^k_{\rho \cup (m)}
 = 2 \sum_{1 \leq i \leq \ell(\rho)} \rho_i b^{k-1}_{\rho \setminus (\rho_{i}) \cup (\rho_{i}+m)} + \sum_{r+s=m \atop r,s \geq 1} b^{k-1}_{\rho \cup (r,s)}  + (m-1) b^{k-1}_{\rho \cup (m)},
 \end{multline*}
 which is exactly what we wanted to prove.
\end{proof}

\begin{remark}
As in section~\ref{SectSymGrpAlg}, define coefficients $d^k_{\rho}$
as solution of the sparse triangular system
\begin{equation}\label{EqLinkBD}
 b^k_{\rho} = \sum_{i = 0}^{m_{1}(\rho)} d^k_{\bar{\rho} \cup 1^i}
 \binom{m_{1}(\rho)}{i}.
 \end{equation}
Then, for a given $k$, only finitely many $d^k_\rho$ are non-zero
(see \cite[Theorem 8.4]{MatsumotoOddJM}).
But, unfortunately, we have no combinatorial interpretation in this case
to obtain directly induction relations on $d$.
This raises the question of the existence of a partial Hecke algebra of $(S_{2n},H_n)$,
out of the scope of this article.
\end{remark}
\begin{remark}
    As in the framework of the symmetric group algebra (paragraph~\ref{SubsectOtherSym}),
    the method extends easily to power-sum symmetric functions.
    More precisely, the following induction relation could be proved with similar arguments:
    \begin{multline*}
        b^{p_{k}}_{\rho \cup (m)} = \delta_{m,1} b^{p_k}_{\rho} +
 2 \sum_{1 \leq i \leq \ell(\rho)} \rho_i
 b^{p_{k-1}}_{\rho \setminus (\rho_{i}) \cup (\rho_{i}+m)}\\
 + \sum_{r+s=m \atop r,s \geq 1} b^{p_{k-1}}_{\rho \cup (r,s)} 
 + (m-1) b^{p_{k-1}}_{\rho \cup (m)} - \delta_{m>1}\ b^{p_{k-1}}_{\rho \cup (m-1)}.
 \end{multline*}
\end{remark}
\subsection{Subleading term}

The induction relation proved in the previous paragraph is a good tool to study the leading and subleading terms of $h_{k}(J_{1}^{(2)},\dots,J_{n}^{(2)}) p_n$, that is the coefficients $b^k_{\rho}$ with $|\rho| - \ell(\rho) = k$ or $k-1$.
Indeed, an immediate induction shows that if the degree condition $|\rho| - \ell(\rho) \leq k$ is not satisfied, then $b^k_{\rho} =0$.
We can also recover the following result proved by S. Matsumoto \cite[Theorem 5.4]{MatsumotoOddJM}.
\begin{proposition}\label{PropDominant}
 If $\rho$ is a partition and $k$ an integer such that $|\rho| - \ell(\rho) = k$, then
 \[b^k_{\rho} = \prod \Cat_{\rho_{i}-1}.\]
\end{proposition}

 But our induction allows us to go further and to compute the subleading term
 (case $|\rho| - \ell(\rho) = k-1$), proving this way a conjecture of S.
 Matsumoto \cite[Conjecture 9.4]{MatsumotoOddJM} corresponding to the case
 where $\rho$ is a hook.
 
 Before stating and proving our result
 (in paragraph~\ref{SubsubsectProofMatsumoto}),
 we need a few definitions and basic lemmas
 on the total area of Dyck paths (paragraph~\ref{SubsubsectArea}).
 
 \subsubsection{Area of Dyck paths}\label{SubsubsectArea}
 
 \begin{definition}
If $I=(i_{1},\dots,i_{r})$ is a weak composition (\textit{i.e.} a sequence of non-negative integers), let us define $\paths_I$ as the set of Dyck paths of length $k=i_{1}+ \dots + i_{r}$ whose height after $i_{1}$, $i_{1}+i_{2}$, \ldots steps is zero
(such a path is the concatenation of Dyck paths of lengths $i_{1}$, $i_2$,\ldots).

If $C$ is a subset of Dyck paths of a given length, denote by $\Area_{C}$
the sum over the paths $c$ in $C$ of the area $\Area_c$ under $c$.
In the case $C=\paths_{I}$,
we shorten the notation and denote $\Area_I=\Area_{\paths_{I}}$.
 \end{definition}
 
For a weak composition $I=(k)$ of length $1$, the set $\paths_I$
is the set of all Dyck paths of length $k$.
 In this case, the area $\Area_k$ has a closed form,
 which has been computed by D. Merlini, R. Sprugnoli, and M. C. Verri
 in \cite{CatalanPathArea}:
 \[\Area_k =  4^k - \binom{2k+1}{k}.\]
 
 The general case can be deduced easily, thanks to the following lemma:
\begin{lemma}
 Let $C_{1}$ and $C_{2}$ be subsets of the set of Dyck paths of length $2m$ and $2n$, respectively.
 Define $C \simeq C_{1} \times C_{2}$ to be the set of Dyck paths of length $2(m+n)$
 which are the concatenation of a path in $C_{1}$ and a path in $C_{2}$. Then 
 \[\Area_{C} = \Area_{C_{1}} \cdot |C_{2}| + |C_{1}| \cdot \Area_{C_{2}}.\]
\end{lemma}
\begin{proof}
    The area under a concatenation $c_1 \cdot c_2$ of two Dyck paths $c_1$
    and $c_2$ is clearly equal to the sum of the areas under $c_1$ and $c_2$.
    Therefore:
    \begin{multline*}
        \Area_C=\sum_{c_1 \in C_1 \atop c_2 \in C_2} \Area_{c_1}+\Area_{c_2}
        =\sum_{c_1 \in C_1 \atop c_2 \in C_2} \Area_{c_1}
        + \sum_{c_1 \in C_1 \atop c_2 \in C_2} \Area_{c_2} \\
    = |C_2| \sum_{c_1 \in C_1} \Area_{c_1} + |C_1|\sum_{c_2 \in C_2} \Area_{c_2}
    = \Area_{C_{1}} \cdot |C_{2}| + |C_{1}| \cdot \Area_{C_{2}}.
    \qedhere \end{multline*}
\end{proof}
A classical result in enumerative combinatorics, perhaps the most widely known,
states that the cardinality of $\paths_{(k)}$ is given by the Catalan number
$\Cat_k := \frac{1}{k+1} \binom{2k}{k}$.

An immediate induction using the lemma above gives the following corollary.
\begin{corollary}
For any weak composition $I$ of length $r$, one has:
 \[\Area_I = \sum_{j=1}^{r} \Area_{i_{j}} \prod_{k \neq j} \Cat_{i_k}. \]
 \end{corollary}
 
One will also need the following induction relation in the next paragraph.
\begin{lemma}[Merlini, Sprugnoli and Verri \cite{CatalanPathArea}]
    \label{LemArea}
 If $m$ is a positive integer, one has:
 \[ \Area_{m-1} = (m-1) \Cat_{m-1} + \sum_{r+s=m \atop r,s \geq 1} \left( \Area_{r-1} \Cat_{s-1} +  \Area_{s-1} \Cat_{r-1} \right).  \]
\end{lemma}
\begin{proof}
This is a consequence of the usual first return decomposition of Dyck paths.
Indeed, let $c$ be a Dyck path of length $2(m-1)$.
We denote $2r$ the $x$-coordinate of the first point where the path touches the 
$x$-axis and $s=m-r$.
Then $c$ is the concatenation of one climbing step, a Dyck path $c_{1}$ of length $2(r-1)$, a down step and a Dyck path $c_{2}$ of length $2(s-1)$ and this decomposition is of course bijective.

\begin{figure}[ht]
\[\includegraphics[width=6cm]{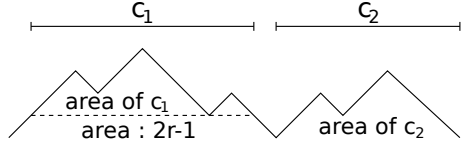}\]
\caption{First-passage decomposition of a Dyck path}
\label{FigCatalan} 
\end{figure}

The area under $c$ is the sum of the areas under $c_{1}$ and $c_{2}$, plus $2r-1$ (see figure~\ref{FigCatalan}). So we write:
\begin{multline*} \Area_{m-1} = \sum_{r+s=m \atop r,s \geq 1} 
\left[ \sum_{c_1 \in \paths_{r-1} \atop c_2 \in \paths_{s-1}}
\Area_{c_1} + \Area_{c_2}+(2r-1) \right] \\
= \sum_{r+s=m \atop r,s \geq 1} 
\left[ \sum_{c_1 \in \paths_{r-1} \atop c_2 \in \paths_{s-1}} \Area_{c_1}
+ \sum_{c_1 \in \paths_{r-1} \atop c_2 \in \paths_{s-1}} \Area_{c_2}
+ \sum_{c_1 \in \paths_{r-1} \atop c_2 \in \paths_{s-1}} (2r-1)  \right] \\
= \sum_{r+s=m \atop r,s \geq 1}                                            
\big[ |\paths_{s-1}| \sum_{c_1 \in \paths_{r-1}} \Area_{c_1} 
+ |\paths_{r-1}| \sum_{c_2 \in \paths_{s-1}} \Area_{c_2}
+ |\paths_{s-1}| \cdot |\paths_{r-1}| \cdot (2r-1) \big] \\
= \sum_{r+s=m \atop r,s \geq 1} \big[ \Area_{r-1} \Cat_{s-1} + 
\Area_{s-1} \Cat_{r-1} + (2r-1) \Cat_{s-1} \Cat_{r-1} \big]. 
\end{multline*}
The last part of the sum may be symmetrized in $r$ and $s$: 
\begin{multline*}
 \sum_{r+s=m \atop r,s \geq 1} (2r-1) \Cat_{s-1} \Cat_{r-1} =  \sum_{r+s=m \atop r,s \geq 1} \frac{1}{2} (2r-1 + 2s -1) \Cat_{s-1} \Cat_{r-1}\\
  = (m-1)  \sum_{r+s=m \atop r,s \geq 1} \Cat_{s-1} \Cat_{r-1} = (m-1) \Cat_{m-1},
\end{multline*} 
which ends the proof of the lemma.
\end{proof}

\subsubsection{Proof of a conjecture of Matsumoto}\label{SubsubsectProofMatsumoto}
Computing the subleading term of $h_{k}(J_{1}^{(2)},\dots,J_{n}^{(2)})\cdot p_n$ 
consists in computing the coefficient $b_\mu^k$ with $k=|\mu| - \ell(\mu)+1$.
Therefore, for a partition $\mu$, we denote:
\[\SD_\mu = b_\mu^{|\mu|-\ell(\mu)+1}.\]
We also denote by $\mu-\bm{1}$ the sequence $\mu_1-1,\mu_2-1,\dots,\mu_{\ell(\mu)}-1$.
As some terms of the sequence can be equal to 0, this is not necessarily a partition,
but it is a weak composition.

\begin{theorem}\label{ThmSubLeading2}
 Let $\mu$ be a partition.  Then
 \[ \SD_{\mu} = \Area_{\mu-\bm{1}}. \]
\end{theorem}

\begin{proof}
 Let $\mu$ be a partition and $k=|\mu| - \ell(\mu)+1$.

 Suppose that we write $\mu = \rho \cup (m)$, for some partition $\rho$ and
 integer $m$.
 We will write Theorem \ref{ThmNewInd2} for $\rho$ and $m$.
\begin{equation}\label{EqNewInd2}
 b^{k}_{\rho \cup (m)} = \delta_{m,1} b^k_{\rho} +
 2 \sum_{1 \leq i \leq \ell(\rho)} \rho_i
 b^{k-1}_{\rho \setminus (\rho_{i}) \cup (\rho_{i}+m)}
 + \sum_{r+s=m \atop r,s \geq 1} b^{k-1}_{\rho \cup (r,s)} 
 + (m-1) b^{k-1}_{\rho \cup (m)}.
\end{equation}

 If $m=1$, the partition $\rho$ fulfills
 \[|\rho| - \ell(\rho) = (|\mu|-1) - (\ell(\mu)-1) = k-1\]
 and thus $b^k_\rho=\SD_\rho$.

 For any $i \leq \ell(\rho)$, the partition 
 $\lambda = \rho \setminus (\rho_{i}) \cup (\rho_{i}+m)$ fulfills
 \[|\lambda|-\ell(\lambda) = |\mu| - (\ell(\mu) -1) = k\]
and therefore, thanks to the degree condition,
$b_{\lambda}^{k-1} = 0.$

In a similar way, as $|\mu| - \ell(\mu)=k-1$, the coefficient $b^{k-1}_\mu$ is 
simply given by Proposition \ref{PropDominant}:
\[ b_{k-1}^\mu = \prod_{i=1}^{\ell(\mu)} \Cat_{\mu_i-1}=\Cat_{m-1}
\prod_{i=1}^{\ell(\rho)} \Cat_{\rho_i-1}. \]

For the last term, for any $r,s \geq 1$ with $r+s=m$,
the partition $\lambda=\rho \cup (r,s)$ fulfills:
\[|\lambda|-\ell(\lambda) = |\mu| - (\ell(\mu) +1) = k-2.\]
Therefore, the coefficient $b_{\lambda}^{k-1}$ corresponds to a subleading
term and $b_{\lambda}^{k-1}=\SD_\lambda$.

Finally, Theorem \ref{ThmNewInd2} becomes in this case:
\begin{equation}\label{EqIndSD}
    \SD_{\rho \cup (m)} = \delta_{m,1} \SD_\rho +
    (m-1) \Cat_{m-1} \prod_{i=1}^{\ell(\rho)} \Cat_{\rho_i-1}
    + \sum_{r,s \geq 1 \atop r+s=m} \SD_{\rho \cup (r,s)}
\end{equation}

This equation gives an induction relation on the coefficient $\SD_\rho$.
We will prove that $\SD_\mu=\Area_{\mu-\bm{1}}$ by a double induction,
first on the size $n$ of the partition $\mu$ and then on the smallest part of
$\mu$

 For $n=1$, one has only the partition $\mu=(1)$ and 
 $\SD_{(1)}=b^1_{(1)}=0=\Area_0$.

Fix now some $n>1$ and suppose that the theorem is true for all partitions of
size smaller than $n$.

If $\mu = \rho \cup (1)$ is a partition of $n$ with smallest part equal to $1$,
then, by equation \eqref{EqIndSD} and the induction hypothesis, one has:
\[\SD_{\mu} = \SD_\rho = \Area_{\rho-1}=\Area_{\mu-\bm{1}}.\]

Let $\mu$ be a partition of $n$ with smallest part $m>1$ and 
suppose that $\SD_\mu=\Area_{\mu-\bm{1}}$ for all partitions of $n$ with smallest
part $m'<m$.
We write $\mu=\rho \cup (m)$ ({\it i.e.} $\rho=\mu \backslash m$).
By equation \eqref{EqIndSD}, 
\[\SD_{\mu} = (m-1) \Cat_{m-1} \prod \Cat_{\rho_{i}-1}
+ \sum_{r+s=m \atop r,s \geq 1} \SD_{\rho \cup (r,s)}. \]
By induction,
\begin{multline*}
\SD_{\rho \cup (r,s)} = \Area_{(\rho \cup (r,s)) -1}
= \Cat_{r-1} \Cat_{s-1} \left(\sum_{i} \Area_{\rho_{i} - 1} \prod_{j \neq i} \Cat_{\rho_{j}-1}\right) \\ +
\Area_{s-1} \Cat_{r-1} \prod_{i} \Cat_{\rho_{i} - 1} + 
\Area_{r-1} \Cat_{s-1} \prod_{i} \Cat_{\rho_{i} - 1} .
\end{multline*}
If we make the substitution in the previous equation, we obtain:
\begin{multline*}
    \SD_{\mu} =\left(\sum_{r,s \geq 1 \atop r+s=m} \Cat_{r-1} \Cat_{s-1} \right) \left(\sum_{i} \Area_{\rho_{i} - 1}
 \prod_{j \neq i} \Cat_{\rho_{j}-1}\right) \\
+ \left((m-1)\Cat_{m-1} + \sum_{r+s=m \atop r,s \geq 1} \big[ \Area_{s-1} \Cat_{r-1} + \Area_{r-1} \Cat_{s-1}\big]\right) \prod_{i} \Cat_{\rho_{i} - 1} .
\end{multline*}
Therefore, using both Lemma~\ref{LemArea} and
the classical induction on Catalan number
$\sum_{r+s=m} \Cat_{r-1} \Cat_{s-1}=\Cat_{m-1}$, one has:
\[\SD_{\mu} = \sum_{i} \Area_{\mu_{i} - 1}
\prod_{j \neq i} \Cat_{\mu_{j}-1} = \Area_{\mu - \bm{1}}. \]
Finally, for any partition $\mu$, one has $\SD_{\mu} =\Area_{\mu-\bm{1}}$, which is 
exactly what we wanted to prove.
\end{proof}

S. Matsumoto established a deep connection between the coefficients $b^{k}_{\mu}$ and the asymptotic expansion of orthogonal Weingarten functions \cite[Theorem 7.3]{MatsumotoOddJM}.
In particular, Theorem~\ref{ThmSubLeading2} gives the subleading term of some matrix integrals over orthogonal group when the dimension of the group goes to infinity.

\section{Towards a continuous deformation?}\label{SectGeneralisation}
The questions studied in sections~\ref{SectSymGrpAlg} and~\ref{SectDoubleClass} may seem quite different at first sight but there exists a continuous deformation from one to the other.

We denote by $\Young_{n}$ the set of Young diagrams (or partitions) of size $n$.
For any $\alpha > 0$, we consider two families of functions on $\Young_{n}$.
\begin{itemize}
  \item First, we call $\alpha$-content of a box of the Young diagram $\lambda$ the quantity $\alpha(j-1) - (i-1)$, where $i$ is its row index and $j$ its column index.
  If $A^{(\alpha)}_{\lambda}$ stands for the multiset of the $\alpha$-contents
  of boxes of$\lambda$, one can look at the evaluation of
  complete symmetric functions $h_{k}(A^{(\alpha)}_{\lambda})$.
  \item Second, we consider Jack polynomials, which is the basis of symmetric function ring indexed by partitions and depending of a parameter $\alpha$ (they are deformations of Schur functions).
  The expansion of Jack polynomials on the power sum basis
  \[J_{\lambda}^{(\alpha)} = \sum_\mu \theta^{(\alpha)}_{\mu}(\lambda) p_{\mu} \]
  defines a family $\theta^{(\alpha)}_{\mu}$ of functions on $\Young_n$
  (we use the same normalization and notation as in \cite[Chapter 6]{McDo} for Jack polynomials).
\end{itemize}

  \begin{proposition}
The functions $\theta^{(\alpha)}_{\mu}$, when $\mu$ runs over the partitions of $n$, form a basis of the algebra $Z_{n}$ of functions over $\Young_{n}$.
\end{proposition}
\begin{proof}
As the cardinality of this family corresponds to the dimension of the space, it is enough to prove that
it spans $Z_n$. Let $f$ be a function on $\Young_n$.

For a fixed $\alpha$, Jack polynomials form a basis of symmetric functions,
therefore there exist some coefficients $d_{\mu,\lambda}^{(\alpha)}$ such that:
\[ p_\mu =\sum_\lambda d_{\mu,\lambda}^{(\alpha)} J_\lambda^{(\alpha)}.\]
Let us define the scalar:
\[ c_\mu=\sum_\lambda d_{\mu,\lambda}^{(\alpha)} f(\lambda).\]
Then one has:
\[
    \sum_\mu c_\mu \theta^{(\alpha)}_{\mu}(\lambda) 
    = \sum_{\mu,\nu} \left( d_{\mu,\nu}^{(\alpha)} \theta^{(\alpha)}_{\mu}(\lambda) \right) f(\nu)
    = f(\lambda),
\]
where the last equality comes from the fact that the matrices $(\theta^{(\alpha)}_{\mu}(\lambda))$
and $(d_{\mu,\lambda}^{(\alpha)})$ are by definition inverse of each other.

Finally, any function $f$ on $\Young_n$ can be written as a linear combination of
$\theta^{(\alpha)}_{\mu}$.
\end{proof}
\begin{remark}
    This proposition is also a consequence of the fact that suitably chosen
    normalizations of $\theta^{(\alpha)}_{\mu}$,
    when $\mu$ runs over all partitions, form a linear basis of the algebra
    of $\alpha$-shifted symmetric functions
    (see \cite[Section 3]{LassalleJackMultirectangular}). However, such
    a sophisticated tool is not needed when $n$ is fixed.
\end{remark}

The proposition implies the existence of some coefficients $a^{k,(\alpha)}_{\mu}$ such that:
\[
h_k(A_\lambda^{(\alpha)}) = \sum_\mu a^{k,(\alpha)}_\mu \theta^{(\alpha)}_\mu(\lambda),
\]

For $\alpha=1$, using the action of Jucys-Murphy element on the Young basis \cite{Jucys1966}
and the discrete Fourier transform of $S_{n}$,
one can see that $a^{k,(1)}_\mu = a^{k}_\mu$.

For $\alpha=2$, using the identification between
Jack polynomials for this special value of the parameter and
zonal polynomials for the Gelfand pair $(S_{2n},H_{n})$ \cite[Chapter 7]{McDo},
as well as the spherical expansion of $h_{k}(J_{1}^{(2)},\dots,J_{n}^{(2)}) p_n$
established by S. Matsumoto \cite[Theorem 4.1]{MatsumotoOddJM},
one has $a^{k,(2)}_\mu = b^{k}_\mu$.
\medskip

It is natural to wonder if there are results similar to Theorems~\ref{ThmNewInd} and~\ref{ThmNewInd2} in the general setting.
Computer exploration using Sage \cite{sage} leads to the following conjecture:
\begin{conjecture}\label{ConjGenAlpha}
 The coefficients $a_\rho^{k,(\alpha)}$ fulfill the linear relation: for any $m \geq 2$,
 \[
a_{\rho \cup (m)}^{k,(\alpha)} = \sum_{r+s=m \atop r,s \geq 1}
a^{k-1,(\alpha)}_{\rho \cup (r,s)} + \alpha \sum_{1\leq i \leq \ell(\rho)}
\rho_i a^{k-1,(\alpha)}_{\rho \backslash \rho_i \cup (\rho_i + m)} +
(\alpha - 1) \cdot (m - 1)\  a_{\rho \cup (m)}^{k-1,(\alpha)}. \label{conj} 
\]
\end{conjecture}

Unfortunately, as we do not have a combinatorial description of the basis 
$\theta_\mu^{(\alpha)}$ in the algebra $Z_n$, we are not able to prove it.
With Lassalle's algebraic approach, one can prove a generalization of Theorem~\ref{ThmLassalle}
(see \cite[Section 11]{LassalleJM}) which is weaker than Conjecture~\ref{ConjGenAlpha}.
Nevertheless, his formula is sufficient to compute inductively the
$a_\rho^{k,(\alpha)}$ and has been used in our numerical exploration. 

In the author's opinion, this conjecture is a hint towards the existence of combinatorial constructions
for other values of the parameter $\alpha$ (like the conjectures of papers
\cite{GouldenJacksonMatchingJackConjecture,LassalleJackMultirectangular,LassalleJackFreeCumulants}).

\section*{Acknowledgments}
 This article has been partially written during a research visit to University of Waterloo.
 The author would like to thank Ian Goulden for his hospitality there.
 He also thanks Michel Lassalle, Sho Matsumoto, Jonathan Novak and Amarpreet Rattan for stimulating discussions on the subject.

\bibliographystyle{alpha}

\bibliography{../courant}

\end{document}